\documentclass[preprint,3p]{elsarticle}

\usepackage{graphicx}
\usepackage{psfrag}
\usepackage{amssymb}
\usepackage{amsthm}
\usepackage{lineno}
\usepackage{amsfonts}
\usepackage{amsmath}
\usepackage{pstricks-add}
\usepackage{bm}
\usepackage{color,hyperref}
\hypersetup{colorlinks,breaklinks,
           linkcolor=blue,urlcolor=blue,
           anchorcolor=blue,citecolor=blue}
\usepackage[caption=false]{subfig}
\usepackage{float}
\usepackage{comment}
\usepackage{amsthm}
\newtheorem{thm}{Theorem}[section]
\newtheorem{lem}{Lemma}[section]
\newtheorem{cor}{Corollary}

\graphicspath{{fig/}}

\journal{Unknown}
\begin{document}
\begin{frontmatter}
\title{Coherence Motivated Sampling and Convergence Analysis of Least-Squares Polynomial Chaos Regression}
\author{Jerrad Hampton}
\author{Alireza Doostan\corref{cor1}}
\ead{alireza.doostan@colorado.edu}
\cortext[cor1]{Corresponding Author: Alireza Doostan}

\address{Aerospace Engineering Sciences Department, University of Colorado, Boulder, CO 80309, USA}

\begin{abstract}
Independent sampling of orthogonal polynomial bases via Monte Carlo is of interest for uncertainty quantification of models, using Polynomial Chaos (PC) expansions. It is known that bounding the spectral radius of a random matrix consisting of PC samples, yields a bound on the number of samples necessary to identify coefficients in the PC expansion via solution to a least-squares regression problem. We present a related analysis which guarantees a mean square convergence using a \textit{coherence parameter} of the sampled PC basis that may be both analytically bounded and computationally estimated. Utilizing asymptotic results for orthogonal polynomials, we bound the coherence parameter for polynomials of Hermite and Legendre type under each respective natural sampling distribution. In both polynomial bases we identify an importance sampling distribution which yields a bound with weaker dependence on the order of the PC basis. For more general orthonormal bases, we propose the {\it coherence-optimal} sampling: a Markov Chain Monte Carlo sampling, which directly uses the basis functions under consideration to achieve a statistical optimality among all such sampling schemes with identical support, and which guarantees recovery with a number of samples that is, up to logarithmic factors, linear in the number of basis functions considered. We demonstrate these different sampling strategies numerically in both high-order and high-dimensional manufactured PC expansions.  In addition, the quality of each sampling method is compared in the identification of solutions to two differential equations, one with a high-dimensional random input and the other with a high-order PC expansion. In all observed cases the coherence-optimal sampling leads to similar or considerably improved accuracy over the other considered sampling distributions.
\end{abstract}
\end{frontmatter}

\section{\texorpdfstring{Introduction}{Introduction}}
\label{sec:intro}

A precise approach to analyzing complex engineering systems requires understanding how various Quantities of Interest (QoI's) depend upon system inputs that are often uncertain. An incomplete understanding may lead to poor design decisions based upon unfounded sureness or incredulity in the QoI's. Uncertainty Quantification (UQ) is an emerging field that aims at addressing these issues in an integral and rigorous manner: it intends for a meaningful characterization of uncertainties from the available information and efficient propagation of these uncertainties for a quantitative validation of model predictions,~\cite{Ghanem91a,LeMaitre10,Xiu10a,Smith13}.

Probability is a natural framework for modeling uncertainty, wherein we assume uncertain inputs are represented by a $d$-dimensional independent random vector $\bm{\Xi}:=(\Xi_1,\cdots,\Xi_d)$ with some joint probability density function $f(\bm{\xi})$. In this manner we model the scalar QoI, e.g., a functional of the solution to a physical system and here denoted by $u(\bm{\Xi})$, as an unknown function of the input, which we seek to approximate. In this work we approximate $u(\bm{\Xi})$, assumed to have a finite variance, using a Fourier-type expansion in multivariate orthogonal polynomials, each of which we denote by $\psi_k(\bm{\Xi})$, referred to as Polynomial Chaos (PC) expansion~\cite{Ghanem91a,Xiu02},
\begin{align}
\label{Eq:PCEDef}
u(\bm{\Xi}) &= \mathop{\sum}\limits_{k=0}^\infty c_k \psi_{k}(\bm{\Xi}).
\end{align}
For computation, we allow an arbitrary number of input dimensions $d$ but assume $u$ can be accurately approximated  in a truncated set of basis of total order less than or equal to $p$. To briefly explain this total order, let $\bm{k} = (k_1,\dots,k_d)$ be a $d\times 1$ multi-index such that $k_i\in{\mathbb{N}\cup\{0\}}$ represents the order of the polynomial $\psi_{k_i}(\Xi_i)$ orthonormal with respect to the probability density of $\Xi_i$. For instance, when $\Xi_i$ follows a uniform or Gaussian distribution, $\psi_{k_i}(\Xi_i)$ are (normalized) Legendre or Hermite polynomials, respectively,~\cite{Xiu02,Soize05}. The $d$-dimensional polynomials $\psi_{\bm{k}}(\bm{\Xi})$ are constructed by the tensorization of $\psi_{k_i}(\Xi_i)$,
\begin{align*}
\psi_{\bm{k}}(\bm{\Xi})=\mathop{\prod}\limits_{i=1}^d\psi_{k_i}(\Xi_i).
\end{align*}
The total order of $p$ implies that we consider all polynomials satisfying
\begin{align*}
\|\bm{k}\|_1\le p\qquad k_i\in{\mathbb{N}\cup\{0\}}\quad\forall i.
\end{align*}
A direct combinatorial count implies that a $d$-dimensional approximation of total order $p$ has $P={p+d\choose d}$ basis polynomials. This total order basis facilitates a polynomial approximation to $u$ of the form
\begin{align}
\label{Eq:PCETrunc}
u(\bm{\Xi}) &\approx \mathop{\sum}\limits_{k=1}^P c_k \psi_{k}(\bm{\Xi})
\end{align}
that favors lower order polynomials. The error introduced by this truncation is referred to as {\it truncation error}, and converges to zero -- in the mean squares sense -- when $c_k = \mathbb{E}[u(\bm\Xi) \psi_{k}(\bm\Xi)]$ and $p\  (\text{hence}\ P)\rightarrow \infty$. Here, $\mathbb{E}$ denotes the mathematical expectation operator.

To identify the PC coefficients $\bm{c}=(c_1,\cdots,c_P)^T$ in (\ref{Eq:PCETrunc}), we consider non-intrusive, i.e., sampling-based, methods where we do not require changes to deterministic solvers for $u$ as we generate realizations of $\bm{\Xi}$ to identify $u(\bm{\Xi})$. We denote these realizations $\bm{\xi}^{(i)}$ and $u(\bm{\xi}^{(i)})$, respectively. We let $i=1:N$ so that $N$ is the number of independent samples considered, and define
\begin{align}
\label{eqn:psi_u}
\bm{u}&:=(u(\bm{\xi}^{(1)}),\cdots,u(\bm{\xi}^{(N)}))^T;\\
\bm{\Psi}(i,j)&:=\psi_{j}(\bm{\xi}^{(i)}),\nonumber
\end{align}
where we refer to $\bm\Psi$ as the {\it measurement matrix}. These definitions imply the matrix equality $\bm{\Psi}\bm{c}=\bm{u}$. We also introduce a diagonal positive-definite matrix $\bm{W}$ such that $\bm{W}(i,i)$ is a function of $\bm{\xi}^{(i)}$ which depends on our sampling strategy and is described in Sections~\ref{sec:motivation} and~\ref{sec:sampling}. To approximate $\bm{c}$ we consider a least squares
 approximation. Specifically, %this involves solving either the $\ell_2$-minimization problem
%
%\begin{align}
%label{eqn:constrained}
%\mathop{\arg\min}_{\bm{c}}\|\bm{c}\|_2\quad \mbox{ subject to }\quad \|\bm{W}\bm{u}-\bm{W}\bm{\Psi}\bm{c}\|_2\le\delta,
%\end{align}
%where $\delta$ is a tolerance of solution inaccuracy due to the truncation error, or the closely related
%
%\begin{align}
%\label{eqn:regularized}
%\mathop{\arg\min}_{\bm{c}}\frac{1}{2}\|\bm{W}\bm{u}-\bm{W}\bm{\Psi}\bm{c}\|_2^2+ \lambda\|\bm{c}\|_2^2,
%\end{align} 
%
%where $\lambda$ is a regularization parameter. We note that the parameters $\lambda$ and $\delta$ are used to both filter and measure error in a practical manner. However,
our theoretical results focus on the solution to the parameter-free optimization problem,
\begin{align}
\label{eqn:lsqr}
\mathop{\arg\min}_{\bm{c}}\|\bm{W}\bm{u}-\bm{W}\bm{\Psi}\bm{c}\|_2,
\end{align}
whose solution $\bm c$ may be computed from the normal equation $(\bm{W\Psi})^T(\bm{W\Psi}) \bm{c}=(\bm{W\Psi})^T\bm{W}\bm{u}$. For most of this work, we assume the truncation error model
\begin{align}
\label{eqn:error}
u(\bm{\Xi})&=\mathop{\sum}\limits_{k=1}^P c_k \psi_{k}(\bm{\Xi}) + \epsilon(\bm{\Xi}),
\end{align}
We often do not express $\epsilon$'s dependence on $\bm{\Xi}$, which should not be confusing when interpreted in context. Additionally, in Section~\ref{subsec:fullErrorthms} we provide theoretical results pertaining to a more complete error model, which is capable of handling interpretations of model or measurement errors.

The least squares regression problem (\ref{eqn:lsqr}) is among major techniques used to generate PC expansions for uncertainty quantification of engineering systems, see, e.g., \cite{Hosder06,Berveiller06,LeMaitre10,CDL13,Migliorati13,Chkifa13,Jones13a,Jones13b,Migliorati14,Gao14,Zhou14,Tang14}. There has been a growing recent interest in understanding the conditions under which problem (\ref{eqn:lsqr}) leads to accurate and stable PC approximation, primarily when the PC basis is of Legendre type and the QoI is sampled according to uniform measure. In particular,~\cite{Migliorati14} presents a bound on the approximation error (in the mean squares sense) for the case of univariate Legendre expansions, provided that the number of samples grows quadratically in the size of the basis. The numeral tests in \cite{Migliorati13} confirm the need for a quadratic growth of sample size, but reveal that a linear growth may be sufficient for accurate Legendre expansions in multiple dimensions when samples are drawn according to the uniform measure. Bounds on the expected mean squares error were derived in \cite{CDL13,Chkifa13} for bounded orthogonal polynomials when sampled according to their natural orthogonality measure. The work in \cite{Zhou14} proposes a deterministic sampling strategy that leads to stable approximations in the Chebyshev basis with a number of samples that scales quadratically in the dimension of the polynomial space. The sampling is also extended to other polynomial bases via a weighted least squares
 formulation. Last but not least, we mention the recent work in \cite{Tang14} that demonstrates empirically the instability of high order Hermite polynomial expansions and proposes an alternative, stable approximation in the basis of Hermite functions. The QoI (in the univariate case) is evaluated at sample points obtained from a nonlinear mapping of random samples uniformly distributed over [-1,1] to the real line. Under this sampling -- which involves parameters to tune -- it is shown that the Hermite function expansion is stable when the number of samples that depends linearly on the number of basis functions.

\subsection{\texorpdfstring{Contributions of This Work}{Contributions of This Work}}
\label{subsec:Contribution}

This work is concerned with the convergence analysis of least squares PC regressions with a particular focus on the role of sampling measure according to which the QoI is evaluated. The novelty of this work arises from utilizing an analysis of Legendre and Hermite polynomials together with the main parameter from~\cite{CDL13} to identify sampling methods for stable and accurate solution recovery via (\ref{eqn:lsqr}), as well as direct practical bounds on the number of samples needed to guarantee recovery. As a result, we present concrete theoretical and practical justifications for when a linear (or quadratic) bound on the number of samples is applicable, as well as a sampling that universally guarantees a linear bound, up to a logarithmic factor. In particular, to our best knowledge, the latter sampling results are the first of their type. Additionally, these analyses highlight the role of the order $p$ and dimension $d$ of the PC expansion in the choice of sampling strategy and quality of solution recovery.

In details, we present results demonstrating the stability of least squares
 solution for {\it low} order, but possibly {\it high} dimensional, Legendre and Hermite expansions when sampled according to their standard orthogonality measure, i.e., uniform and standard Gaussian, respectively. Such a stability is guaranteed for a number of samples that scales linearly on the size of the basis set.   

Additionally, we provide a contribution of particular practical interest in that we also consider sampling {\it low} dimensional, but possibly {\it high} order, Hermite polynomials uniformly over a $d$-dimensional ball -- with a radius depending on the order of approximation -- instead of sampling from the standard Gaussian measure. This sampling of Hermite polynomials is analogous to Hermite function expansion, as in \cite{Tang14}, when the QoI is appropriately weighted. The uniform sampling of the present work is substantially different from the sampling strategy proposed in \cite{Tang14}. We provide analytic and numeric results justifying the use of this {\it importance sampling} distribution for the recovery of Hermite PC expansions. In a similar context, sampling of Legendre polynomials according to Chebyshev measure will be discussed. 

Finally, we analytically identify an importance sampling distribution with a statistical {\it optimality}, in terms of the {\it spectral radius} of the realized matrix $\bm{\Psi}$ as a key recovery parameter of the method, and identify a Monte Carlo Markov Chain sampler for which we provide associated numeric results. This approach, here referred to as {\it coherence-optimal} sampling, provides a general sampling scheme for the reconstruction of expansions in general bases. This sampling scheme allows us to guarantee recovery with a number of samples that scales linearly in the number of basis functions, up to a logarithmic factor, and may be extended to PC types other than Legendre and Hermite. We highlight that such a linear scaling of the number of samples makes the least squares
 regression approach a competitive framework for UQ of systems where the QoI is not cheap to evaluate.

The presentation in this work has Section~\ref{sec:motivation} providing necessary background information and motivating our approach, while Section~\ref{sec:sampling} describes our sampling methods and provides key theoretical results. Section~\ref{sec:examples} demonstrates the performance of the sampling methods, and Section~\ref{sec:Proofs} presents the proofs to the Theorems from Section~\ref{sec:sampling}.
\section{\texorpdfstring{Review, Definitions and Recovery Results}{Review, Definitions and Recovery Results}}
\label{sec:motivation}
To place the results presented in Section~\ref{sec:cdl_theorems} in context, we provide a summary of recent results on least squares
 PC regression, while introducing two coherence definitions that are fundamental to our theoretical results, and in constructing our sampling methods.
\subsection{\texorpdfstring{Sampling Definitions}{Sampling Definitions}}
\label{subsubsec:sampling}
We consider as a basis the set of orthonormal polynomials, $\{\psi_k(\bm{\xi})\}_{k=1}^{P}$, and define $B(\bm{\xi})$ to be
\begin{align}
\label{eqn:btspec}
B(\bm{\xi}):=\sqrt{\mathop{\sum}\limits_{k=1}^{P}|\psi_k(\bm{\xi})|^2}.
\end{align}
This represents a uniformly least upper bound on the sum of the squares of the basis polynomials of interest, and is motivated by work in~\cite{CDL13}, which uses this function to generate a condition-type number useful for guaranteeing recovery. A similar approach is taken in~\cite{CandesPlan,Hampton14} for the case of solutions recovered via $\ell_1$-minimization. A bound on $B(\bm{\xi})$ may be attained from
\begin{align}
\label{eqn:btbound}
B^2(\bm{\xi})\le P \mathop{\sup}\limits_{k=1:P}|\psi_k(\bm{\xi})|^2,
\end{align}
where $\mathop{\sup}\limits_{k=1:P}|\psi_k(\bm{\xi})|^2$ may be bounded for several orthogonal polynomials~\cite{RauhutWard,Szego,Hermite1,AskeyWainger,LagAsym,Jacobi,Hampton14}. In addition, to ease analysis of $B(\bm{\xi})$, formed of piecewise continuous polynomials, we consider, possibly smooth, functions $G(\bm{\xi})$
\begin{align}
\label{eqn:GDef}
G(\bm{\xi})\ge B(\bm{\xi}) \qquad \forall\bm{\xi}\in\Omega,
\end{align}
representing an $L_2(\Omega,f)$-integrable upper bound on $B(\bm{\xi})$, for all $\bm{\xi}\in\Omega$, where $\Omega$ denotes the set of possible realizations of $\bm{\Xi}$.

In this case, we have that $\psi_k(\bm{\xi})/G(\bm{\xi})\le 1$. It follows that for any set $\mathcal{S}\subseteq\Omega$,
\begin{align}
\label{eqn:optimal_pdf_gen}
f_{\bm{Y}}(\bm{\xi}):=c^2f(\bm{\xi})G^2(\bm{\xi}),
\end{align}
with
\begin{align}
\label{eqn:normalizingconstant}
c = \left(\int_{\mathcal{S}}f(\bm{\xi})G^2(\bm{\xi})d\bm{\xi}\right)^{-1/2},
\end{align}
is a probability distribution supported on $\mathcal{S}$, which we consider as the distribution for $\bm{Y}$, a copy of $\bm\Xi$ \cite{Hampton14}. Let $\delta_{i,j}$ denote the Kronecker delta such that $\delta_{i,j}=1$ if $i=j$ and $0$ if $i\ne j$. Note that for $i,j=1:P$,
\begin{align}
\label{Eqn:TransformIntegral}
\left|\int_{\mathcal{S}}\frac{\psi_i(\bm{\xi})}{cG(\bm{\xi})}\frac{\psi_j(\bm{\xi})}{cG(\bm{\xi})}c^2f(\bm{\xi})G^2(\bm{\xi})d\bm{\xi}-\delta_{i,j}\right|&\le \epsilon_{i,j},
\end{align}
and we may select $\mathcal{S}$ such that $\epsilon_{i,j}$ may be made as small as needed, e.g., if we take $\mathcal{S} = \Omega$, then $\epsilon_{i,j}=0$ due to the orthonormality of $\{\psi_k(\bm\Xi)\}_{k=1}^{P}$. For this purpose we employ the heuristic of selecting $\mathcal{S}$ to encompass the largest values of $f(\bm{\xi})$ until $\mathcal{S}$ is large enough to satisfy the condition (\ref{Eqn:CoherenceUnbounded}), discussed in Section~\ref{subsubsec:Coherence}. The justification for this is that in unbounded domains, e.g., for Hermite polynomials, regions of small $f(\bm{\xi})$ typically correspond to larger $B(\bm{\xi})$ as $p$ grows~\cite{Szego,AskeyWainger,Hermite1}.

Notice that $\{\psi_k(\bm Y)\}_{k=1}^{P}$ is no longer (approximately) orthonormal under $f_{\bm Y}$. Following (\ref{Eqn:TransformIntegral}), we therefore consider the weight function 
\begin{align}
\label{eqn:weightFunction}
w(\bm{Y})&:=\frac{1}{cG(\bm{Y})},
\end{align}
so that $\{w(\bm{Y})\psi_k(\bm{Y})\}_{k=1}^P$ are (approximately) orthonormal random variables. This function defines the diagonal positive-definite matrix $\bm{W}$ from (\ref{eqn:lsqr}) as
\begin{align*}
 \bm{W}(i,i)= w(\bm{\xi}^{(i)}),
\end{align*}
where $\bm{\xi}^{(i)}$ is the $i$th realization of $\bm{Y}$. For the interest of notation, we refer to all realized random vectors by $\bm{\xi}$ regardless of the sampling distribution for $\bm{\xi}$, noting that the weight function $w$ depends on that distribution. Additionally, we note that for simulation, we are typically not interested in the normalizing constant $c$ associated with the sampling distribution.

\subsection{\texorpdfstring{Coherence Parameter Definition}{Coherence-Parameter Definition}}
\label{subsubsec:Coherence}
Consider realizations of $w(\bm{Y})\psi_k(\bm{Y})$ for $k=1:P$. We investigate the {\it coherence} parameter 
\begin{align}
\label{Eqn:CoherenceBounded}
\mu_{2}(\bm{Y}) &:= \mathop{\sup}\limits_{\bm{\xi}\in\Omega}\mathop{\sum}_{k=1}^{P}|w(\bm{\xi})\psi_k(\bm{\xi})|^2,
\end{align}
which, following~\cite{CDL13}, yields a bound on the realized spectral radius of $\bm{W\Psi}$ and useful results pertaining to the solution of (\ref{eqn:lsqr}). We often suppress the dependence of $\mu_2$ on the sampling scheme, as there should be no confusion that it depends on how the random variables are sampled. From (\ref{eqn:weightFunction}) we are motivated to consider $G(\bm{\xi})$ taken to be $B(\bm{\xi})$ as defined in (\ref{eqn:btspec}), and as we shall
show in Theorem~\ref{thm:LowPTransformedSamples} that this sampling method leads to $\mu_2=P$ and a linear dependence on the number of samples needed as compared to the number of basis functions considered up to logarithmic factors. Fortunately, asymptotic (in $p$) results on the orthogonal polynomials considered here give choices for $G(\bm{\xi})$ that, while generally lead to sub-optimal $\mu_2$, correspond to sampling from known measures $f_{\bm Y}$, as discussed in Section \ref{subsec:asym}.

We utilize the definition in (\ref{Eqn:CoherenceBounded}) when analyzing Legendre polynomials which are bounded on the domain $[-1,1]^d$. However, we note that (\ref{Eqn:CoherenceBounded}) is not useful when $\sup_{k=1:P,\bm{\xi}\in\Omega}|w(\bm{\xi})\psi_k(\bm{\xi})|^2$ is infinite, such as when $\psi_k(\bm{\xi})$ are Hermite polynomials and $w(\bm{\xi}) = 1$. If $N$ is the number of samples of $\bm{Y}$ which we will take, motivated by the truncation in \cite{CandesPlan,Hampton14} for recovery via $\ell_1$-minimization, we consider a truncation of $\Omega$ to a subset $\mathcal{S}$ and let
\begin{align}
\label{Eqn:CoherenceEll2Unbounded}
\mu_{2}(\bm{Y}) &:= \mathop{\min}\limits_{\mathcal{S}}\left\{\mathop{\sup}\limits_{\bm{\xi}\in\mathcal{S}}\mathop{\sum}\limits_{k=1}^P|w(\bm{\xi})\psi_k(\bm{\xi})|^2,\cdots\right.\\
\nonumber
\mbox{subject to }&\left.\mathbb{P}(\mathcal{S}^c)\le\frac{1}{NP};\ \mathop{\sum}\limits_{k=1}^{P}\mathbb{E}\left(|w(\bm{Y})\psi_k(\bm{Y})|^2\bm{1}_{\mathcal{S}^{c}}\right)\le\frac{1}{20}P^{-1/2}\right\},
\end{align}
where $\mathcal{S}$ is a subset of the support of $f$, the superscript $c$ denotes a set complement, and $\bm{1}$ is an indicator function. If $\mathcal{S}$ is a finite set, then $\mu_{2}(\bm{Y})$ is finite for polynomial basis functions, as we consider here. \\[-.3cm]

\noindent{\bf Remark.} We note that the conditions in (\ref{Eqn:CoherenceEll2Unbounded}) can be tweaked, with tradeoffs concerning the probability of failure of (\ref{eqn:lsqr}) in finding a solution $\bm c$ and a bias in spectral norm of $(\bm{W\Psi})^T(\bm{W\Psi})$ as discussed in Section~\ref{sec:Proofs}. Also, for the sampling described in Section~\ref{subsec:MCMC}, we can forego this truncation. %\\[-.3cm]
\
We also consider the same truncation to  $\mathcal{S}$ of (\ref{Eqn:CoherenceEll2Unbounded}) in regards to the coherence parameter $\mu_{\infty}(\bm{Y})$ defined by
\begin{align}
\label{Eqn:CoherenceUnbounded}
\mu_{\infty}(\bm{Y}) &:= \mathop{\sup}\limits_{\bm{\xi}\in\mathcal{S},k=1:P}|w(\bm{\xi})\psi_k(\bm{\xi})|^2.
\end{align}
This parameter is often easier to directly bound than the sum associated with $\mu_2$. Notice that $\mu_\infty$ may be used to bound $\mu_2$ via
\begin{align}
\label{Eqn:CohBound}
\mu_2(\bm{Y})\le P \mu_{\infty}(\bm{Y}).
\end{align}
This bound may be sharp, for example when $|w(\bm{\xi})\psi_k(\bm{\xi})|^2=1$ uniformly in $\bm{\xi}$ and $k$, as occurs when using basis functions restricted to have complex modulus one for $\bm{\xi}\in\mathcal{S}$. Additionally, one has that $\mu_2\ge \mu_{\infty}$, which may also be sharp if, for instance, the basis functions have non-overlapping support. However, these cases do not apply to the polynomial bases considered here. In fact, as has been investigated in~\cite{Hampton14}, and will be reproduced in Section~\ref{sec:sampling} in regards to polynomial bases, with an appropriate sampling distribution $\mu_{\infty}$ is most heavily influenced by $p$, the order of PC approximation, or $d$, the dimension of approximation. 

As seen in the next section, following \cite{CDL13}, solution recovery from ({\ref{eqn:lsqr}) can be guaranteed with a number of samples that is proportional to $\mu_2$ up to logarithmic factors in $P$. This is bounded by a number of samples that is proportional to $\mu_{\infty}P\log(P)$. The sampling distribution $f_{\bm Y}$ will imply the coherence $\mu_{\infty}$ and its dependence on $d$ and $p$. In particular, in the case of Legendre polynomials, $\mu_{\infty}$ grows exponentially in $p$ under standard sampling but is bounded in $d$. Similarly, $\mu_{\infty}$ grows exponentially in $d$ under Chebyshev sampling, but is bounded in $p$. See for example the results of~\cite{Hampton14}, whose relevant theorems for bounds on $\mu_{\infty}$ are presented in Section~\ref{sec:sampling}. We note that this analysis can be directly used to explain observations regarding solution recovery from ({\ref{eqn:lsqr}) using a number of samples that depends either linearly or quadratically on the number of basis functions $P$, see, e.g., \cite{Chkifa13,Migliorati14,Gao14,Tang14}.
\subsection{\texorpdfstring{Convergence Theorems}{Convergence Theorems}}
\label{sec:cdl_theorems}
The following theorems use the coherence parameter in either (\ref{Eqn:CoherenceBounded}) or (\ref{Eqn:CoherenceEll2Unbounded}) to bound the number of samples necessary to recover the PC coefficients from ({\ref{eqn:lsqr}) with high probability. Our first theorem concerning solution recovery is comparable to Theorem 2 in~\cite{CDL13}, and demonstrates a quality mean squared convergence with high probability. 
\begin{thm}
\label{thm:SampleDepth}
Let 
\begin{align*}
\hat{u}(\bm\Xi)&=\mathop{\sum}\limits_{k=1}^P\hat{c}_k\psi_k(\bm{\Xi}),
\end{align*}
where $\hat{\bm{c}}=(\hat{c}_1,\cdots,\hat{c}_P)^T$ is the solution to (\ref{eqn:lsqr}).
It follows that for $\mathcal{E}$ a sampling event (defined in Section~\ref{sec:Proofs}) that occurs with probability at least $1-1/P-2P\exp(-0.1N\mu_2^{-1})$ (or $1-2P\exp(-0.1N\mu_{2}^{-1})$ when the definition in (\ref{Eqn:CoherenceBounded}) is used),
\begin{equation*}
\label{eqn:mse_bound_1}
\mathbb{E}\left(\Vert u-\hat{u}\Vert^2_{L_2(\Omega,f)}; \mathcal{E}\right)\le \mathbb{E}(\epsilon^2)\left(1+\frac{4\mu_{2}}{N}\right),
\end{equation*}
where 
\begin{align}
\label{eqn:restricted_expectation}
\mathbb{E}(X;\mathcal{E}) = \int_{\mathcal{E}}X(\bm{\xi})f(\bm{\xi})d\bm{\xi} =\mathbb{E}(X|\mathcal{E})\mathbb{P}(\mathcal{E})
\end{align}
denotes the expectation restricted to the event (also known as restricted expectation), and is closely related to conditional expectation.
\end{thm}

\noindent{\bf Remark.} We note that expectations in this theorem -- as well as in the subsequent theorems -- are with respect to the original measure $f(\bm\xi)$ of $\bm\xi$, even when we sample according to $f_{\bm{Y}}$, see the discussion in Section \ref{subsec:proofEll2TruncOnly}. \\[-.3cm]

The Theorem \ref{thm:SampleDepth} and its dependence on $\mu_{2}$ motivate us to sample such that $\mu_2$ is small. Indeed, if $\mu_{2}$ scales linearly in $P$, as for example under the sampling of Section~\ref{subsec:MCMC} or when $\mu_\infty$ is bounded by a constant, then we can take $N$ to scale nearly linearly with $P$ in a way that depends on the desired accuracy. Noting that $\Vert u-\hat{u}\Vert^2_{L_2(\Omega,f)}-\mathbb{E}(\epsilon^2)\ge 0$ and applying the Markov inequality, gives a corollary that is also useful in bounding a failure probability for a particular least-squares computation.
\begin{cor}
\label{cor:MarkovBound}%
For any $\tau>0$,
\begin{align*} 
\mathbb{P}\left(\Vert u-\hat{u}\Vert^2_{L_2(\Omega,f)}\ge \mathbb{E}(\epsilon^2) +\tau \right)\le \frac{1}{P}+2P\exp\left(-\frac{0.1N}{\mu_2}\right)+\frac{4\mu_2}{N}\cdot\frac{\mathbb{E}(\epsilon^2)}{\tau}.
\end{align*}
Further, when the definition in (\ref{Eqn:CoherenceBounded}) is used for $\mu_2$, this bound holds without the $1/P$ term.
\end{cor}
Additionally, the following theorem useful for bounding needed sample sizes follows directly from Theorem~\ref{thm:SampleDepth},
\begin{thm}
\label{thm:numSamples}
For any $\tau>0$, we seek to identify the number of samples, $N_\tau$, so that with the same definitions as in Theorem~\ref{thm:SampleDepth}, it follows that
\begin{align*}
\mathbb{E}\left(\Vert u-\hat{u}\Vert^2_{L_2(\Omega,f)}; \mathcal{E}\right)\le \mathbb{E}(\epsilon^2)+\tau.
\end{align*}
where $\mathcal{E}$ occurs with probability larger than $1-1/P-2P\exp(-0.1N\mu_{2}^{-1})$ (or $ 1-2P\exp(-0.1N\mu_{2}^{-1})$ if the definition in (\ref{Eqn:CoherenceBounded}) is used). Stated another way, we wish to bound the number of samples so that we are within $\tau$ of the theoretically minimal mean squared error with high probability. This is achieved by allowing the number of samples to satisfy
\begin{align*}
\tau N_\tau &\ge 4\mathbb{E}(\epsilon^2)\mu_{2}.
\end{align*}
\end{thm}
We see that if $\mu_{2}$ from (\ref{Eqn:CoherenceEll2Unbounded}) is such that $\mu_2=\nu P$ for some constant $1\le\nu \le \mu_\infty$ depending on the measure of $\bm Y$, then $N_\tau$ scales linearly in $P$, the number of basis functions considered. This is the case for example, when we sample as described in Section~\ref{subsec:MCMC}, or when $\mu_\infty$ is bounded by a constant. If we make this assumption that $\mu_2=\nu P$, then from the result of Theorem \ref{thm:numSamples}, the failure probability $1-\varrho$ is bounded by 
\begin{align*} 
 1-\varrho \le \exp\left(\log(2P)-\frac{0.1N}{\nu P}\right).
\end{align*}
An algebraic rearrangement of terms gives that
\begin{align*} 
N\le 10\nu P\log\left(\frac{2P}{1-\varrho}\right).
\end{align*}
This provides an explicit bound on the number of samples needed to guarantee an accurate recovery within a certain failure probability. We summarize this implication and that of error control suggested by Theorem~\ref{thm:numSamples} in the following corollary. 
\begin{cor}
\label{cor:SamplesExplicit} 

Suppose that $\mu_2$ is defined as in (\ref{Eqn:CoherenceBounded}) and is bounded by $\nu P$ for some  $0<\nu \le\mu_{\infty}$ depending on the measure of $\bm Y$. Then, to insure recovery to within $\tau$ of minimum mean-squared-error with probability at least $\varrho$, we may take
\begin{align*} 
 N_{\tau,\varrho} \ge \nu\cdot \max\left\{\frac{4\mathbb{E}(\epsilon^2)}{\tau}\cdot P,10P\cdot\log\left(\frac{2P}{1-\varrho}\right)\right\},
\end{align*}
which scales linearly in $P$ up to logarithmic factors. If $\mu_2$ is defined as in (\ref{Eqn:CoherenceEll2Unbounded}) and $\varrho$ is chosen so that $\varrho < 1 - 1/P$, then
\begin{align*}
 N_{\tau,\varrho} \ge \nu\cdot \max\left\{\frac{4\mathbb{E}(\epsilon^2)}{\tau}\cdot P, 10P\cdot\log\left(\frac{2P}{1-1/P-\varrho}\right)\right\},
\end{align*}
which scales similarly.
\end{cor}
We note that the separation of the two sample size requirements within the $\max$ operator of Corollary \ref{cor:SamplesExplicit} is fruitful in that each is sharp with regards to its derivation in the order of samples required.

The left term in the maximum is related to recovery when a quality sampling has been achieved. For general $\ell_2$ recovery with $P$ unknowns, because of the the existence of a non-trivial null-space with as many as $P-1$ samples, a linear growth in the number of samples is the minimum order that is anticipated for this term.

The right term in the maximum is related to identifying that a quality sampling has occurred with probability $\varrho$. The logarithmic term for this bound is also tight in certain contexts, and is related to the {\it coupon collector} problem (that is anticipating how many draws from a uniform distribution on $\{1,\cdots,P\}$ are needed to have a draw from each of $\{1,\cdots,P\}$). We refer the interested reader to Remark 5.6 of~\cite{Tropp12a}.

This decoupling of constants is useful for bounding the accuracy of a computed solution while guaranteeing a certain probability of successful recovery, or similarly, identifying a probability of successful recovery while guaranteeing a certain accuracy in the computed solution. 

\subsection{\texorpdfstring{A More Robust Noise Model}{A More Robust Noise Model}}
\label{subsec:fullErrorthms}

In this section we consider $u$ as a function of $(\bm{\Xi},\bm{\chi})$, where $\bm{\chi}$ are unobserved random variables that may have parameters which are functions of realized $\bm{\xi}$, and are useful to account for model and measurement errors. As we do not observe $\bm{\chi}$, we do not attempt to compute an estimate $\hat{u}$ that depends on $\bm{\chi}$ and instead attempt to identify the conditional random variable,
\begin{align*}
u(\bm{\Xi}) &:= \mathbb{E}\left(u(\bm{\Xi},\bm{\chi})\big|\bm{\Xi}\right),
\end{align*}
noting that there should be no confusion in the notation when function arguments are present. We then consider the error model,
\begin{align}
u(\bm{\Xi},\bm{\chi})&:=u(\bm{\Xi})+\epsilon_M(\bm{\chi}|\bm{\Xi});\\ 
u(\bm{\Xi})&:=\mathop{\sum}\limits_{k=1}^P c_k \psi_{k}(\bm{\Xi}) + \epsilon_T(\bm{\Xi}),
\end{align}
where $\epsilon_T$ is still interpreted as a truncation error that depends only on $\bm{\Xi}$, and $\epsilon_M(\bm{\chi}|\bm{\Xi})$. This model gives us the following theorem, which is comparable to but significantly different than Theorem 3 in~\cite{CDL13}.

\begin{thm}
\label{thm:SampleDepthNoise}
Using the same definitions as in Theorem~\ref{thm:SampleDepth}, it follows that
\begin{align*}
\mathbb{E}\left(\Vert u-\hat{u}\Vert^2_{L_2(\Omega,f)}; \mathcal{E}\right)\le \mathbb{E}(\epsilon_T^2)+ 8\mu_{2}\frac{\mathbb{E}(\epsilon_T^2)+\mathbb{E}(\epsilon_M^2)}{N}.
\end{align*}
\end{thm}
This theorem demonstrates that the recovery in this case is quite similar to that of Theorem~\ref{thm:SampleDepth}. Applying the Markov inequality similarly gives a bound on failure probability in this case.
\begin{cor}
\label{cor:MarkovBoundNoise}
For any $\tau>0$,
\begin{align*} 
\mathbb{P}\left(\Vert u-\hat{u}\Vert^2_{L_2(\Omega,f)}\ge \mathbb{E}(\epsilon^2) +\tau \right)\le \frac{1}{P}+2P\exp\left(-\frac{0.1N}{\mu_2}\right)+\frac{8\mu_2}{N}\cdot\frac{\mathbb{E}(\epsilon^2)+\mathbb{E}(\epsilon_M^2)}{\tau}.
\end{align*}
Further, when the definition in (\ref{Eqn:CoherenceBounded}) is used for $\mu_2$, this bound holds without the $1/P$ term.
\end{cor}
Additionally, the number of samples required to guarantee an accurate recovery in expectation is also similar to that of Theorem~\ref{thm:numSamples}, stated in the following theorem, which follows directly from Theorem~\ref{thm:SampleDepthNoise}.
\begin{thm}
\label{thm:numSamplesFullErr}
For any $\tau>0$, we seek to identify the number of samples, $N_\tau$, so that with probability at least $1-1/P-2P\exp(-0.1N\mu_{2}^{-1})$ (or $1-2P\exp(-0.1N\mu_{2}^{-1})$ if the definition in (\ref{Eqn:CoherenceBounded}) is used),
\begin{align*}
\mathbb{E}\left(\Vert u-\hat{u}\Vert^2_{L_2(\Omega,f)}; \mathcal{E}\right)\le \mathbb{E}(\epsilon_T^2)+ \tau.
\end{align*}
Stated another way, we wish to bound the number of samples so that we are within $\tau$ of the theoretically minimal mean squared error with high probability. This is achieved by allowing the number of samples to satisfy
\begin{align*}
\tau N_\tau &\ge 8\mu_{2}\left(\mathbb{E}(\epsilon_T^2)+\mathbb{E}(\epsilon_M^2)\right).
\end{align*}
\end{thm}
We can make a similar explicit bound on the number of samples for this noise model as well.
\begin{cor}
\label{cor:SamplesExplicit2} 
Suppose that $\mu_2$ is defined as in (\ref{Eqn:CoherenceBounded}) and is bounded by $\nu P$ for some $0\le\nu \le\mu_{\infty}$ depending on the measure of $\bm Y$. Then to insure recovery to within $\tau$ of minimum mean-squared-error with probability at least $\varrho$, we may take
\begin{align*} 
 N_{\tau,\varrho} \ge \nu\cdot\max\left\{\frac{8\left(\mathbb{E}(\epsilon_T^2)+\mathbb{E}(\epsilon_M^2)\right)}{\tau}\cdot P,10P\cdot\log\left(\frac{2P}{1-\varrho}\right)\right\},
\end{align*}
which scales linearly in $P$ up to logarithmic factors. If $\mu_2$ is defined as in (\ref{Eqn:CoherenceEll2Unbounded}) and $\varrho$ is chosen so that $\varrho < 1 - 1/P$, then
\begin{align*}
 N_{\tau,\varrho} \ge \nu\cdot \max\left\{\frac{8\left(\mathbb{E}(\epsilon_T^2)+\mathbb{E}(\epsilon_M^2)\right)}{\tau}\cdot P,10P\cdot\log\left(\frac{2P}{1-1/P-\varrho}\right)\right\},
\end{align*}
which scales similarly.
\end{cor}
\noindent{\bf Remark.} We note that the same considerations as followed Corollary~\ref{cor:SamplesExplicit} apply here in this slightly more general model, with the only difference being a doubling of the first multiplicative constant.

\section{\texorpdfstring{Sampling Methods}{Sampling Methods}}
\label{sec:sampling}
Here, we describe the sampling methods that we consider in this work, and present theorems related to recovery when we use them. Our objective is to investigate the implication of certain choices of sampling measure $f_{\bm Y}$ on $\mu_2$, and hence on the parameter $\nu$ that controls the bound on sampling rates in Corollaries \ref{cor:SamplesExplicit} and \ref{cor:SamplesExplicit2}. 

We first consider a sampling according to random variables defined by the orthogonality measure in Section~\ref{subsec:std}. Such a sampling, dubbed here  {\it standard sampling}, is commonly used in PC regression, \cite{Hosder06,LeMaitre10,Doostan11a,Mathelin12a}. We also consider sampling from a distribution associated with an asymptotic analysis, although not necessarily corresponding with the asymptotic distribution, of the orthogonal polynomials $\psi_k(\bm\Xi)$ in Section~\ref{subsec:asym}, and refer to it as {\it asymptotic sampling}. For the above two sampling methods, we discuss the dependence of $\mu_\infty$, an upper bound on $\nu$, on the dimensions $d$ of inputs and total order $p$ of the PC basis. Finally, in Section~\ref{subsec:MCMC}, we introduce the {\it coherence-optimal sampling} that corresponds to minimizing $\nu$, as opposed to $\mu_\infty$, directly. The empirical results of Section \ref{sec:examples} suggest that sampling strategies with (considerably) smaller $\mu_2$, lead also to improved solution recoveries, hence signifying the role of the coherence parameter $\mu_2$ in the actual performance of the least squares
PC regression. We note that the proofs for all but the last of the theorems in this section may be found with no changes within~\cite{Hampton14}. %Jerrad: Removed an errant 'equivalently' as there is no equivlence between \mu_\infty and \mu_2 as parameters. In fact I removed the mention of \nu and \mu_\infty there

\subsection{\texorpdfstring{Standard Sampling}{Standard Sampling}}
\label{subsec:std}

We first consider sampling $\bm{\xi}$ according to the orthogonality measure $f(\bm\xi)$ of the PC bases, for which we set the weights $w(\bm\xi) =1$. For the $d$-dimensional Legendre polynomials the standard method corresponds to sampling from the uniform distribution on $[-1,1]^d$, while for $d$-dimensional Hermite polynomials this corresponds to samples from a multi-variate normal distribution such that each of $d$ coordinates is an independent standard normal random variable. 

For the standard sampling of orthogonal polynomials of arbitrary, total order $p$ in dimension $d$, the following results hold for $\mu_\infty$ (a bound on the parameter $\nu$ in theorems of Section \ref{sec:motivation}). 

\begin{thm}
\label{thm:NatSampleCoherence}
Assume that $d=o(p)$, that is, $d$ is asymptotically dominated by $p$. Additionally, let $N = O(P^k)$ for some $k>0$, that is the number of samples does not grow faster than a polynomial in the number of basis polynomials considered. For $d$-dimensional Hermite polynomials of total order $p\ge 1$, the coherence in (\ref{Eqn:CoherenceUnbounded}) is bounded by 
\begin{align}
\label{eqn:HerNatBound}
\mu_{\infty}(\bm{\Xi})&\le C_p\cdot \eta_p^p,
\end{align}
for some constants $C_p,\eta_p$ depending on $p$. For $d = o(p)$, and as $p\rightarrow\infty$, we may take $C_p$ and $\eta_p$ to be larger than but arbitrarily close to $1$ and $\exp(2-\log(2))\approx 3.6945$, respectively.

A standard sampling of the $d$-dimensional Legendre polynomials of total order $p$ gives a coherence of
\begin{align}
\label{eqn:LegNatBound}
\mu_{\infty}(\bm{\Xi})&\le \exp(2p).
\end{align}
\end{thm}

As shown in~{\cite{Hampton14}}, for the case of $p<d$ the bound in (\ref{eqn:LegNatBound}) may be improved to $\mu_{\infty}(\bm{\Xi})\le 3^p \approx \exp(1.1p)$. Additionally, for $p>d$, the bound in (\ref{eqn:LegNatBound}) is loose, but a sharper dimension-dependent bound is given by $(2p/d+1)^d$.\\[-.3cm]

\noindent{\bf Remark.} When sampling Hermite polynomials, we have the technical requirement that $N=O(P^k)$ for some finite $k$. We note that this includes the sample sizes suggested by the corollaries in Section~\ref{sec:motivation}.\\[-.3cm]

We note that the combination of the bounds $\nu\le \mu_\infty$ and (\ref{eqn:HerNatBound}), together with Corollary~\ref{cor:SamplesExplicit}, implies that the number of samples required for recovery from Hermite polynomials may grow exponentially with the total order of approximation.

\subsection{\texorpdfstring{Asymptotic Sampling}{Asymptotic Sampling}}
\label{subsec:asym}

We next consider taking $G(\bm{\xi})$ to be an asymptotic envelope for the polynomials as the order $p$ goes to infinity. For very large $p$, it then follows that $G(\bm{\xi})$ is a good approximation to $B(\bm{\xi})$.

For $d$-dimensional Hermite polynomials, an analysis that utilizes Hermite Functions and presented in~\cite{Hampton14} leads to a choice of $G(\bm\xi)$ corresponding} to sampling uniformly from within the $d$-dimensional ball of radius $\sqrt{2}\sqrt{2p+1}$. This leads to a weight function given by
\begin{align*}
w(\bm{\xi}):= \exp(-\|\bm{\xi}\|_2^2/4).
\end{align*}

Following \cite{Hampton14}, the algorithm below may be utilized to sample uniformly from the $d$-dimensional ball of radius $r$. First, let $\bm{Z}:=(Z_1,\cdots,Z_d)$ be a vector of $d$ independent normally distributed random variables with zero mean and the same variance. If $U$ is another independent random variable that is uniformly distributed on $[0,1]$, then
\begin{align*}
\bm{Y}&:=\frac{\bm{Z}}{\|\bm{Z}\|_2}rU^{1/d},
\end{align*}
represents a random sample uniformly distributed within the $d$-dimensional ball of radius $r$.

For the $d$-dimensional Legendre polynomials the asymptotic sampling measure is known coordinate-wise as the Chebyshev distribution on $[-1,1]^d$,~\cite{RauhutWard}, that is the distribution in each of $d$ coordinates is
\begin{align*}
f_Y(\xi)&:=\frac{1}{\pi\sqrt{1-\xi^2}},
\end{align*}
for $\xi\in[-1,1]$. Each coordinate is easily simulated from $\cos(\pi U)$ where $U$ is uniformly distributed on $[0,1]$. Additionally, this leads to a weight function given by
\begin{align*}
 w(\bm{\xi}):=\mathop{\prod}\limits_{i=1}^d(1-\xi_i^2)^{1/4}.
\end{align*}
In the following theorem we bound the coherence parameter $\mu_\infty$ associated with these alternative sampling of Hermite and Legendre polynomials, and demonstrate a weaker asymptotic dependence of $\mu_\infty$ on $p$..
\begin{thm}
\label{thm:TransformedSamples}
Assume that $N = O(P^k)$ for some $k>0$, that is, the number of samples does not grow faster than a polynomial in the number of basis polynomials considered. We note that this includes the sample sizes suggested by the Theorems in Section~\ref{sec:motivation}. Let $V(r,d)=(r\sqrt{\pi})^d/\Gamma(d/2+1)$ denote the volume inside the hypersphere with radius $r$ in dimension $d$. 

For the sampling of Hermite polynomials, sampling uniformly from the $d$-dimensional ball of radius $\sqrt{2}\sqrt{(2+\epsilon_{p})p+1}$, and weighting realized $\psi_k(\bm{\xi}^{(i)})$ on this ball by $w(\bm\xi^{(i)}) = \exp(-\|\bm{\xi}^{(i)}\|_2^2/4)$, gives
\begin{equation}
\mu_{\infty}=O(\pi^{-d/2}V(\sqrt{2p},d))=O((2p)^{d/2}/\Gamma(d/2+1)).
\end{equation}
Here, we note that $\epsilon_{p}\rightarrow 0$ if $d = o(p)$, and that the radius of the sampling is a factor of $\sqrt{2}$ times larger than the radius of the volume in the coherence, due to a normalization explained in~\cite{Hampton14}.

For the sampling of $d$-dimensional Legendre polynomials according to the $d$-dimensional Chebyshev distribution and weight $\psi_k(\bm{\xi})$ proportional to $w(\bm\xi) = \prod_{i=1}^d(1-\xi_i^2)^{1/4}$, regardless of the relationship between $d$ and $p$, we have that
\begin{align}
\label{eqn:LegTransformedSample}
\mu_{\infty}&\le 3^d.
\end{align}

\end{thm}

We note that for Legendre polynomials sampled by Chebyshev distribution we have a complete independence of the order of approximation, which agrees with previous results in~\cite{RauhutWard}. We also highlight that the results of Theorems \ref{thm:NatSampleCoherence} and \ref{thm:TransformedSamples} suggest sampling Hermite and Legendre polynomials from the standard distribution when $d>p$ and from the asymptotic distribution when $d<p$. A similar observation has been made in \cite{Yan12}. \\[-.3cm]

\noindent{\bf Remark.} In the numerical results of this work, for uniform sampling of Hermite polynomials, we set $\epsilon_{p}$ in Theorem~\ref{thm:TransformedSamples} to be zero, leaving as an open problem the determination of an optimal $\epsilon_{p}$, and hence sampling radius.

\subsection{\texorpdfstring{Coherence-optimal Sampling}{Coherence-optimal Sampling}}
\label{subsec:MCMC}

We finally consider taking $G(\bm{\xi})=B(\bm{\xi})$ in (\ref{eqn:optimal_pdf_gen}), which implies sampling $\bm{\xi}$ according to the distribution
\begin{equation}
\label{eqn:optimal_pdf}
f_{\bm Y}(\bm \xi)=c^2f(\bm{\xi})B^2(\bm{\xi}),
\end{equation}
with $c^2 = \left(\int_{\mathcal{S}}f(\bm{\xi})B^2(\bm{\xi})d\bm{\xi}\right)^{-1}$, and where $f(\bm{\xi})$ is the measure with respect to which $\{\psi_k(\bm\xi)\}_{k=1}^{P}$ are naturally orthogonal. Corresponding to this sampling, we apply the weight function
\begin{align*}
w(\bm{\xi})=\frac{1}{cB(\bm{\xi})},
\end{align*}
or $w(\bm{\xi})=1/B(\bm{\xi})$ in practice.
\subsubsection{\texorpdfstring{Sampling via Markov Chain Monte Carlo}{Sampling via Markov Chain Monte Carlo}}
\label{subsubsec:MCMCmethod}

To avoid computing the normalization constant $c$ in (\ref{eqn:optimal_pdf}), we sample $\bm{\Xi}$ from (\ref{eqn:optimal_pdf}) using a Monte Carlo Markov Chain (MCMC) approach, specifically via the Metropolis-Hastings sampler~\cite{Hastings}, which requires evaluations of $B(\bm{\xi})$ from (\ref{eqn:btspec}). Additionally, this sampling distribution allows the easy evaluation of $w(\bm{\xi})$ using only the realized samples.

In more detail, the MCMC sampler requires a candidate distribution and when $p\le d$, we suggest those obtained from Section~\ref{subsec:std}, giving a standard normal sampling for Hermite polynomials, and sampling uniformly for $[-1,1]^d$ for Legendre polynomials. Similarly, when $p>d$ we suggest those obtained from Section~\ref{subsec:asym}, giving a uniform sampling on a $d$-dimensional ball for Hermite polynomials, and $d$-dimensional Chebyshev sampling for Legendre polynomials. Note that each proposal distribution covers the entire domain $\mathcal{S}$, and if the proposal and target distribution approximately match, then the acceptance rate is high and few burn-in samples are needed to approximately draw from the desired distribution for $\bm{Y}$. One caveat which we note is that the proofs of Theorems~\ref{thm:SampleDepth} and~\ref{thm:SampleDepthNoise} require independent sampling, so that it is proper to restart a chain after each accepted sample, but a more practical method is to discard intermediate samples so that serial dependence is small,~\cite{MCMCBook}. We note that in applications where evaluation of the QoI is expensive to evaluate, the generation of the samples, $\{\bm{\xi}^{(i)}\}_{i=1}^N$, is not typically a bottleneck, so that the extra cost for quality MCMC sampling is frequently acceptable in practice. Theorem~\ref{thm:LowPTransformedSamples} justifies the intuition that taking $G(\bm{\xi})$ associated with sampling to be the envelope function $B(\bm{\xi})$ leads to the minimal $\nu=1$, or equivalently $\mu_{2} = P$.

\begin{thm}
\label{thm:LowPTransformedSamples} 
Let $\mathcal{S}$ be a set chosen to satisfy the conditions of (\ref{Eqn:CoherenceEll2Unbounded})
implying that no subset $\mathcal{S}_s$ of $\mathcal{S}$ with $\mu_{2}(\mathcal{S}_s)<\mu_{2}(\mathcal{S})$ satisfies the conditions of (\ref{Eqn:CoherenceEll2Unbounded}).
Let $B(\bm{\xi})$ be as in (\ref{eqn:btspec}). If we sample from the distribution proportional to $f(\bm{\xi})B^2(\bm{\xi})$ and weight $\psi_k(\bm{\xi})$ proportional to $w(\bm\xi) =1/B(\bm{\xi})$, then the coherence parameter $\mu_2$ achieves a minimum over all sampling schemes of $\psi_{k}(\bm{\xi})$, $k=1:P$, and distributions supported on $\mathcal{S}$.

Further if $\mathcal{S}$ is chosen to be $\Omega$ then for this sampling strategy, $\mu_2 = P$. 
\end{thm}
\section{Numerical Examples}
\label{sec:examples}

We next study the empirical performance of different sampling schemes discussed in Section~\ref{sec:sampling}, first exploring an estimate of the coherence parameters $\mu_2$ in Section~\ref{subsec:cohEstimates}. We then investigate the recovery of randomly generated functions in Section~\ref{subsec:RandomSignals}, and functions given by the solution to an ODE with random coefficient modeling the amount of reaction at a given time in an adsorption model from \cite{Makeev02} in Section~\ref{subsec:ODE}. Additionally, we consider a high-dimensional model of a thermally driven cavity flow in Section~\ref{subsec:Heat}.
\subsection{Estimates of Coherence $\mu_2$}
\label{subsec:cohEstimates}
The coherence parameter $\mu_2$ of Section~\ref{subsubsec:Coherence} can be estimated using a large sample of realized $\psi_k(\bm\xi)$. Doing so leads to the results in Figure~\ref{Fig:HermCoh} for Hermite polynomials, and those in Figure~\ref{Fig:LegCoh} for Legendre polynomials. We consider three sampling schemes, the standard scheme where we sample based on the underlying distribution of random variables in question as in Section~\ref{subsec:std}; an asymptotically motivated method to insure a coherence with weaker dependence on the order $p$ as in Section~\ref{subsec:asym}; and a coherence-optimal sampling based on the distribution proportional to the envelope of basis functions as in Section~\ref{subsec:MCMC}. We observe that standard sampling tends to perform poorly at high orders, while asymptotic sampling tend to perform poorly for high-dimensional problems. Additionally, coherence-optimal sampling performs well in all regimes. These observations are consistent with the theoretical results presented in Section \ref{sec:sampling}.

\begin{figure}[h!]
\centering
\includegraphics[scale=0.40]{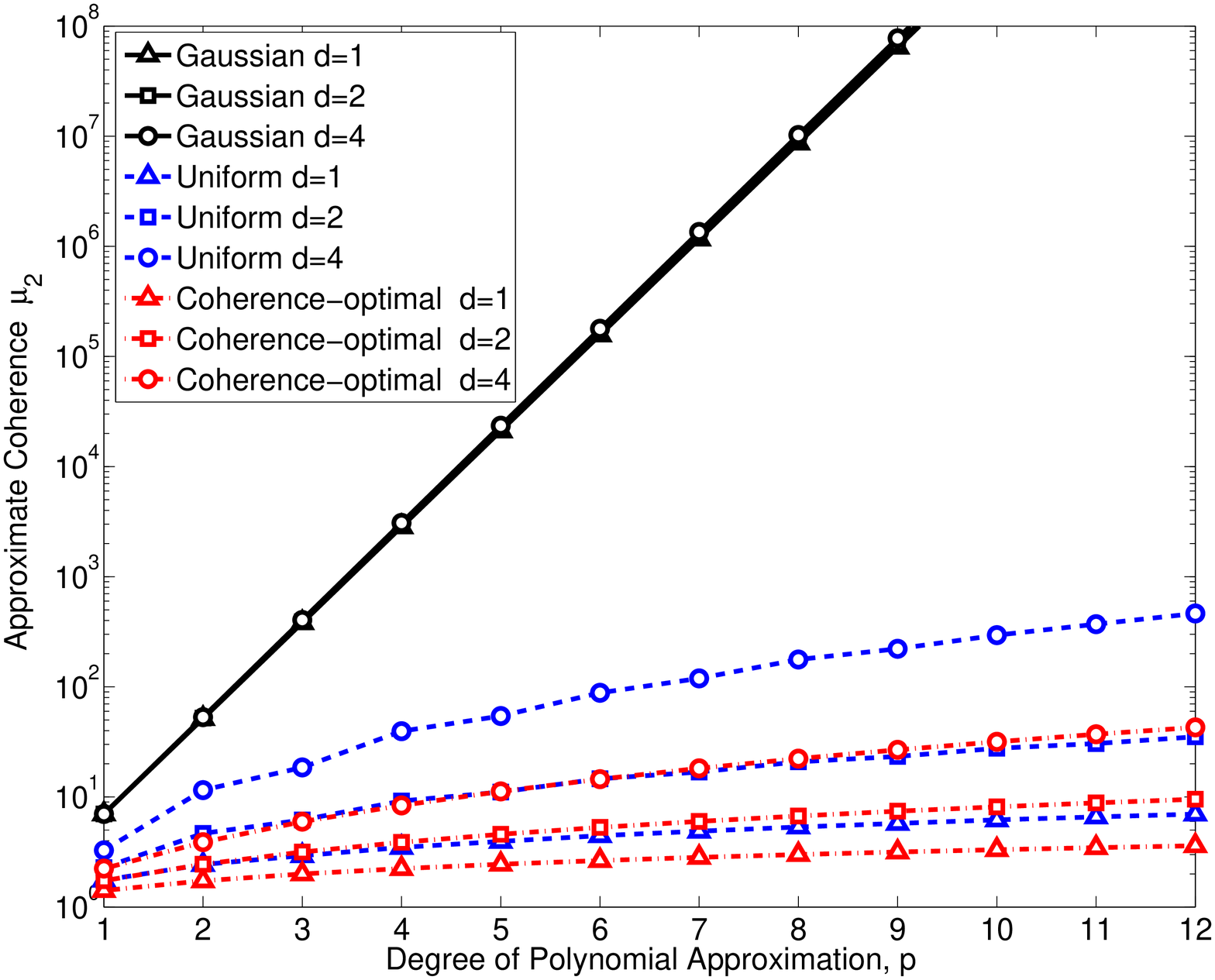}%\includegraphics[trim=15mm 10mm 10mm 10mm, clip, scale=0.45]{CohHerm.eps}
\caption{Computed $\mu_2$ for different sampling methods of Hermite polynomials for different $d$ and $p$.}
\label{Fig:HermCoh}
\end{figure}
\begin{figure}[h!]
\centering
\includegraphics[scale=0.40]{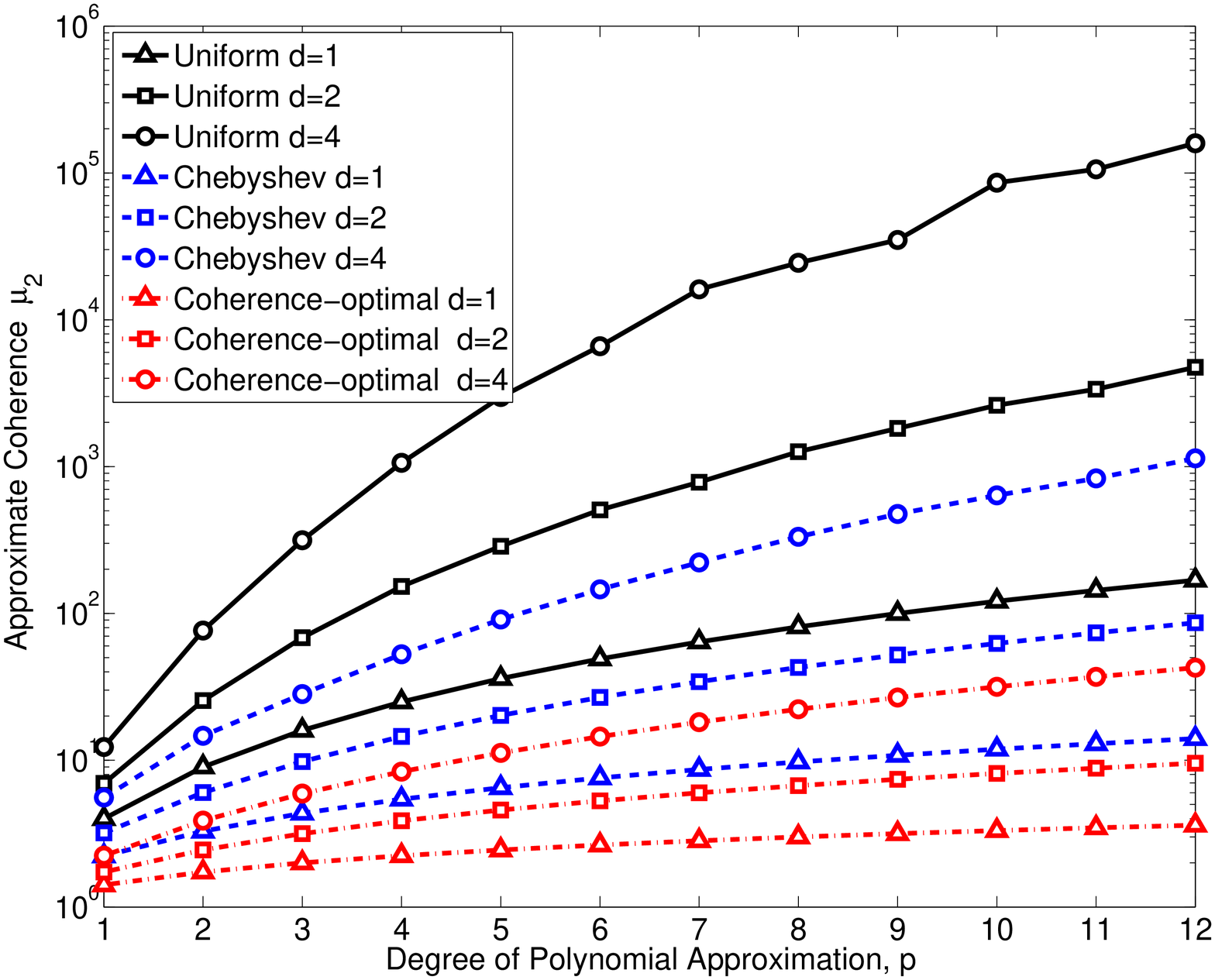}%\includegraphics[trim=15mm 10mm 10mm 10mm, clip, scale=0.45]{CohLeg.eps}
\caption{Computed $\mu_2$ for different sampling methods of Legendre polynomials for different $d$ and $p$.}
\label{Fig:LegCoh}
\end{figure}
\subsection{Manufactured Functions}
\label{subsec:RandomSignals}

In this section, we investigate the reconstruction accuracy of the competing sampling schemes on randomly generated solution vectors, $\bm{c}$, such that $\bm{\Psi}\bm{c}=\bm{u}$. Here, each coordinate of $\bm{c}$ is an independent standard normal random variable, which implies that, unlike for most solutions of practical interest, the coordinates of $\bm{c}$ do not exhibit a decay with respect to the order of $\psi_k(\bm\xi)$. We measure reconstruction accuracy as a function of the number of independent samples of $\bm{Y}$, denoted by $N$. We declare $\hat{\bm{c}}$ to successfully recover $\bm{c}$ if $\|\hat{\bm{c}}-\bm{c}\|_2/\|\bm{c}\|_2\le 0.02$, where $\hat{\bm{c}}$ is a solution to (\ref{eqn:lsqr}) and in this work is computed using MATLAB's {\tt mldivide} function, which is more ubiquitiously known as the backslash ($\backslash$) matrix operation. Each success probability is calculated from an average of a large number of independent realizations of $\bm\Psi$ and $\bm c$. We consider two cases, one in which there is no noise in the evaluation of $\bm{u}$, and one in which there is independent, normally distributed additive noise in the evaluations of {\color{black}$u(\bm{\xi}^{(i)})$, each having standard deviation $0.03\cdot |u(\bm{\xi}^{(i)})|$}.

The results in Figures~\ref{Fig:HermRand} and~\ref{Fig:LegRand} identify a sharp transition in the ability of least squares problem (\ref{eqn:lsqr}) to recover $\bm c$ {\color{black}in the noisy cases}. In particular, we highlight the following notable observations: for the high order case $(d,p)=(2,30)$, the standard sampling fails to recover the solution, while the asymptotic sampling succeeds. For the high-dimensional case $(d,p)=(30,2)$, the standard sampling is better than the asymptotic sampling, and for the moderate values of $(d,p)=(5,5)$ the coherence-optimal method outperforms both asymptotic and standard sampling, while the standard sampling of Hermite expansion performs poorly. In all cases, the coherence-optimal sampling leads to recovery that is similar to those of the other two sampling strategies or provides considerable improvements. {\color{black}For the noisy cases, the recovery of the sampling strategies examined here is similar expect for the Hermite expansion with $(d,p)=(2,30)$ where the standard sampling fails to converge within the range of sample sizes considered.}

\begin{figure}[h!]
%\centering
\hspace{-.7cm}
\includegraphics[trim = 30mm 0mm 0mm 0mm, clip, scale=0.5]{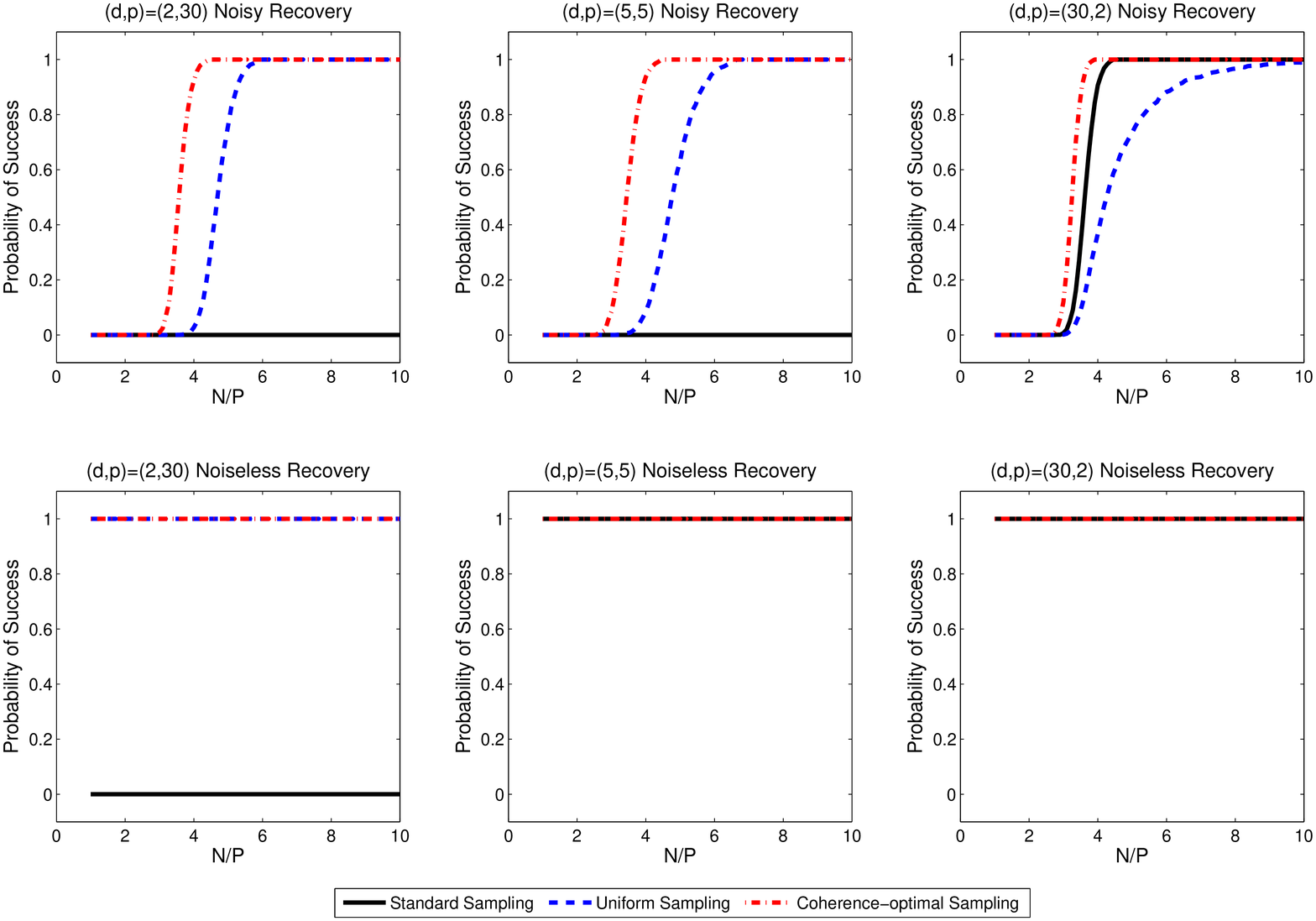}
\caption{Hermite recovery probability as a function of sample size $N$: Top graphs are for noisy recovery, while bottom are for noiseless recovery. Horizontal position corresponds to differing $(d,p)$.}
\label{Fig:HermRand}
\end{figure}

\begin{figure}[h!]
%\centering
\hspace{-.7cm}
\includegraphics[trim = 30mm 0mm 0mm 0mm, clip, scale=0.5]{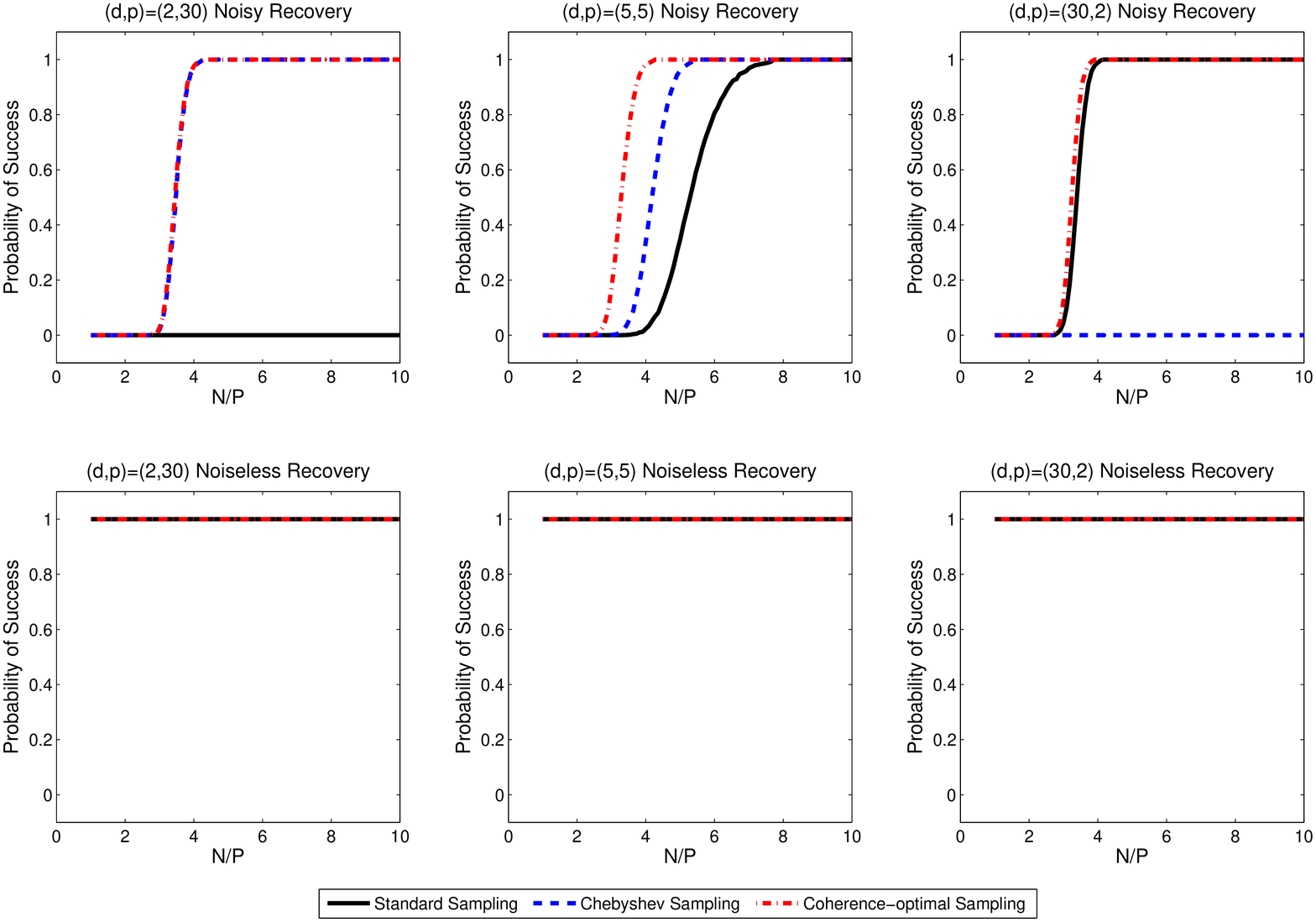}
\caption{Legendre recovery probability as a function of sample size $N$: Top graphs are for noisy recovery, while bottom are for noiseless recovery. Horizontal Position corresponds to differing $(d,p)$.}
\label{Fig:LegRand}
\end{figure}

\subsection{A Surface Reaction Model with Random Input}
\label{subsec:ODE}

The second problem of interest in this work is to quantify the uncertainty in the solution $\rho$ of the non-linear evolution equation  
\begin{align}
\label{eqn:reaction}
\left\{\begin{array}{l}\frac{d\rho}{dt} = \alpha(1-\rho) - \gamma\rho - \kappa(1-\rho)^2\rho, \\ 
\rho(t=0) = 0.9,\end{array}\right.
\end{align}
which models the surface coverage of certain chemical species, as examined in \cite{Makeev02,Lemaitre04b,Hampton14}. We consider uncertainty in the adsorption, $\alpha$, and desorption, $\gamma$, coefficients, and model both as shifted log-normal random variables. Specifically, we assume
\begin{align*}
\alpha &= 0.1 + \exp(0.05\ \Xi_1),\\
\gamma &= 0.001 + 0.01\exp(0.05\ \Xi_2),
\end{align*}
where $\Xi_1, \Xi_2$ are independent standard normal random variables; hence, the dimension of our random input is $d=2$. The reaction rate constant $\kappa=10$ in (\ref{eqn:reaction}) is assumed to be deterministic. 

Our QoI is $\rho_c:=\rho(t=4,\Xi_1,\Xi_2)$, and we consider a Hermite PC expansion of total order $p=32$, giving $P=561$ basis functions to approximate $\rho_c$. This high-order approximation is necessary due to the large gradient of $\rho_c$ in terms of the random variables, as evidenced by the relatively slow decay of coefficients in the reference solution presented in Figure~\ref{Fig:RefSolODE}. This reference solution is computed using Gauss-Hermite quadrature approximation of the PC coefficients.  

In Figure~\ref{Fig:ODERMSE}, we see plots of moments for the relative root-mean-squared error -- between the reference and least squares solutions -- as a function of the number of samples, $N$, obtained from 200 independent replications for each $N$. We find that the standard sampling fails to converge, while the asymptotic and coherence-optimal samplings lead to converged solution as $N$ is increased. {\color{black}We note that, in all cases, the standard sampling resulted in rank-deficient matrices $\bm\Psi$, and the reported results were generated using MATLAB's {\tt lsqr} function with a large number of iterations.}

\begin{figure}[h!]
\centering
\includegraphics[trim = 0mm 0mm 0mm 0mm, clip, scale=0.6]{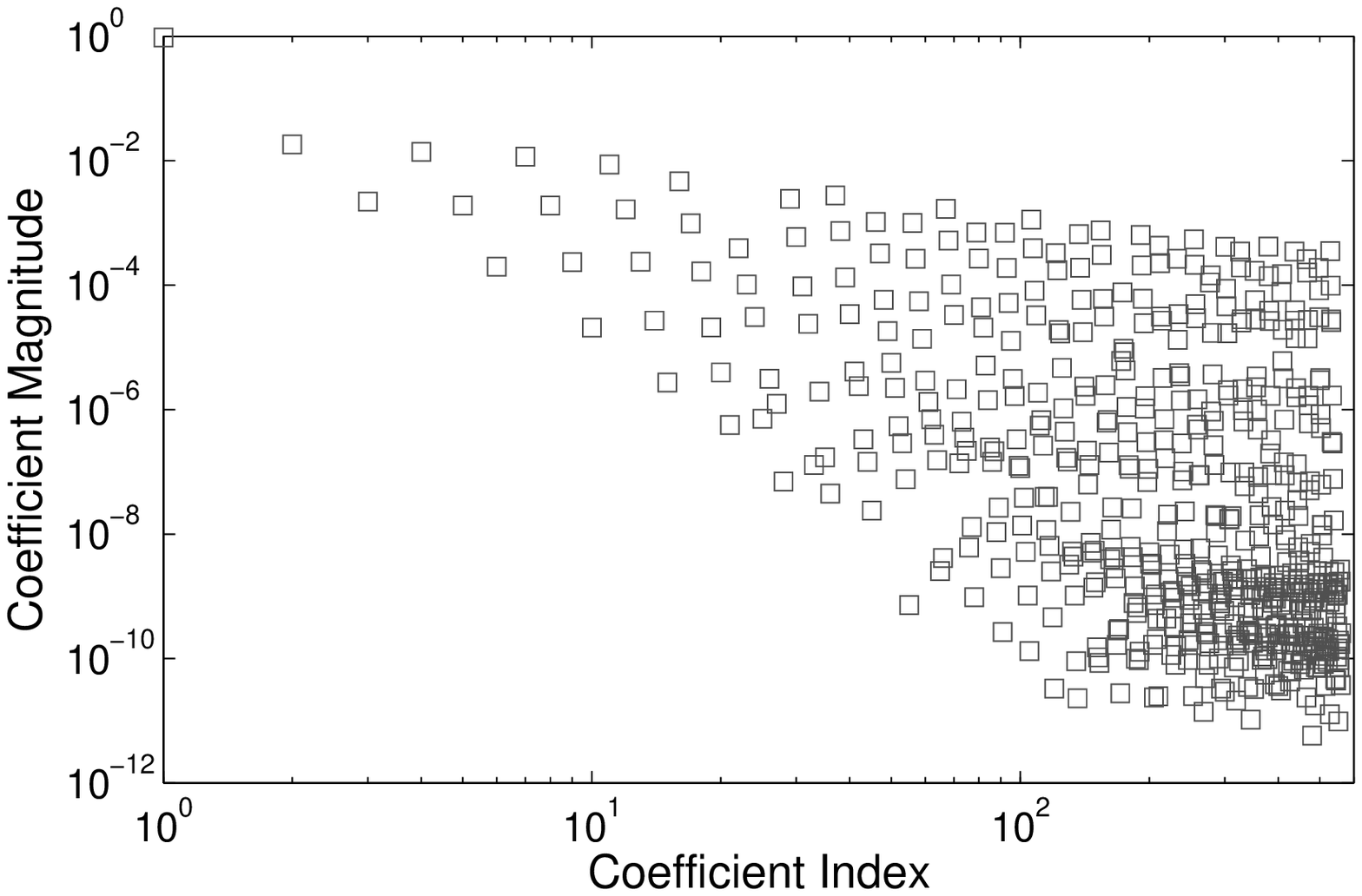}
%\vspace{2cm}
\caption{PC coefficients of reference solution for the QoI $\rho_c:=\rho(t=4,\Xi_1,\Xi_2)$ in the surface reaction model.}
\label{Fig:RefSolODE}
\end{figure}

\begin{figure}[h!]
\hspace{-.5cm}
\begin{minipage}{.5\textwidth}
\centering
\includegraphics[scale=0.35]{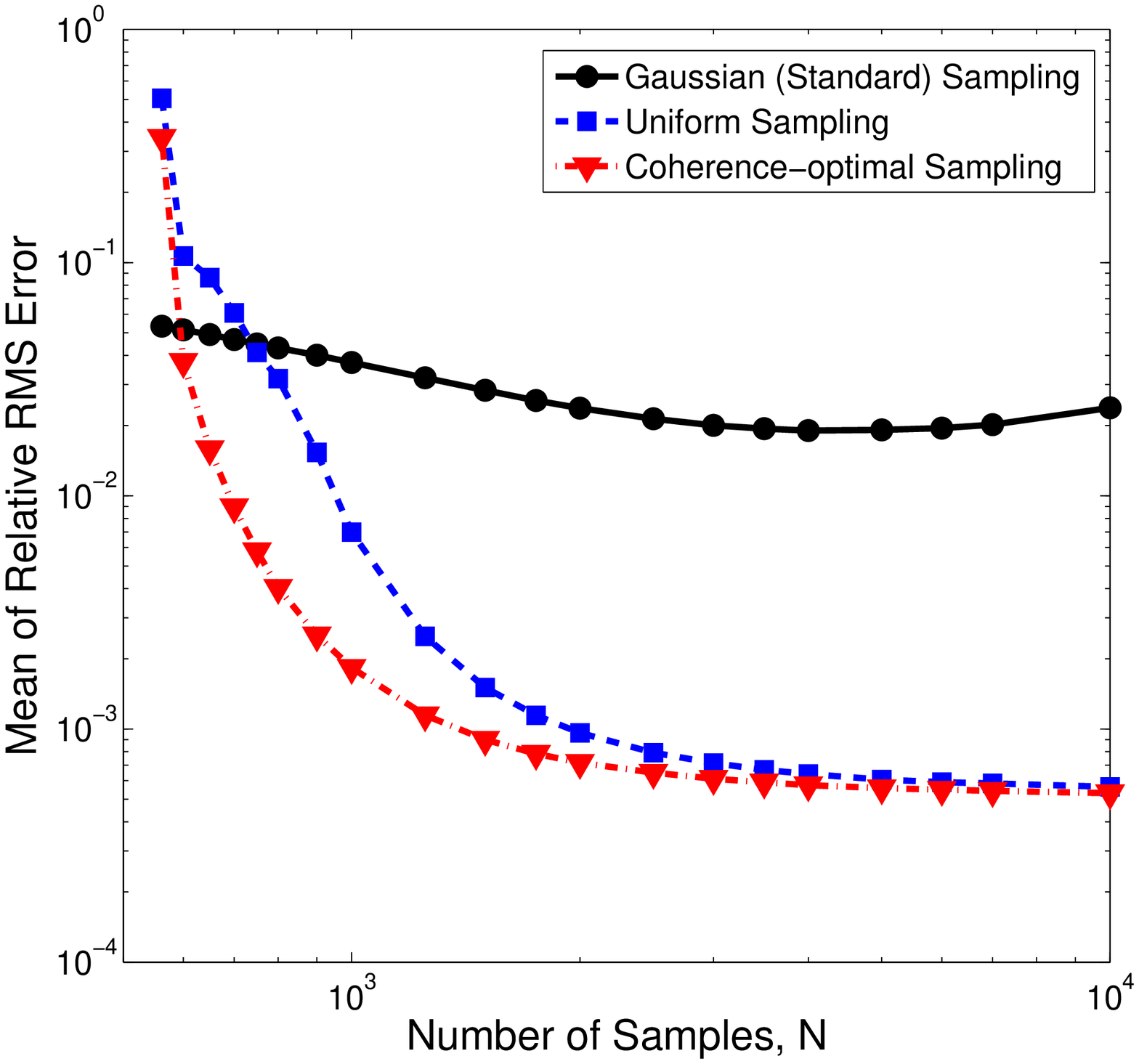}
\end{minipage}
\hspace{.1cm}
\begin{minipage}{.5\textwidth}
\centering
\includegraphics[scale=0.35]{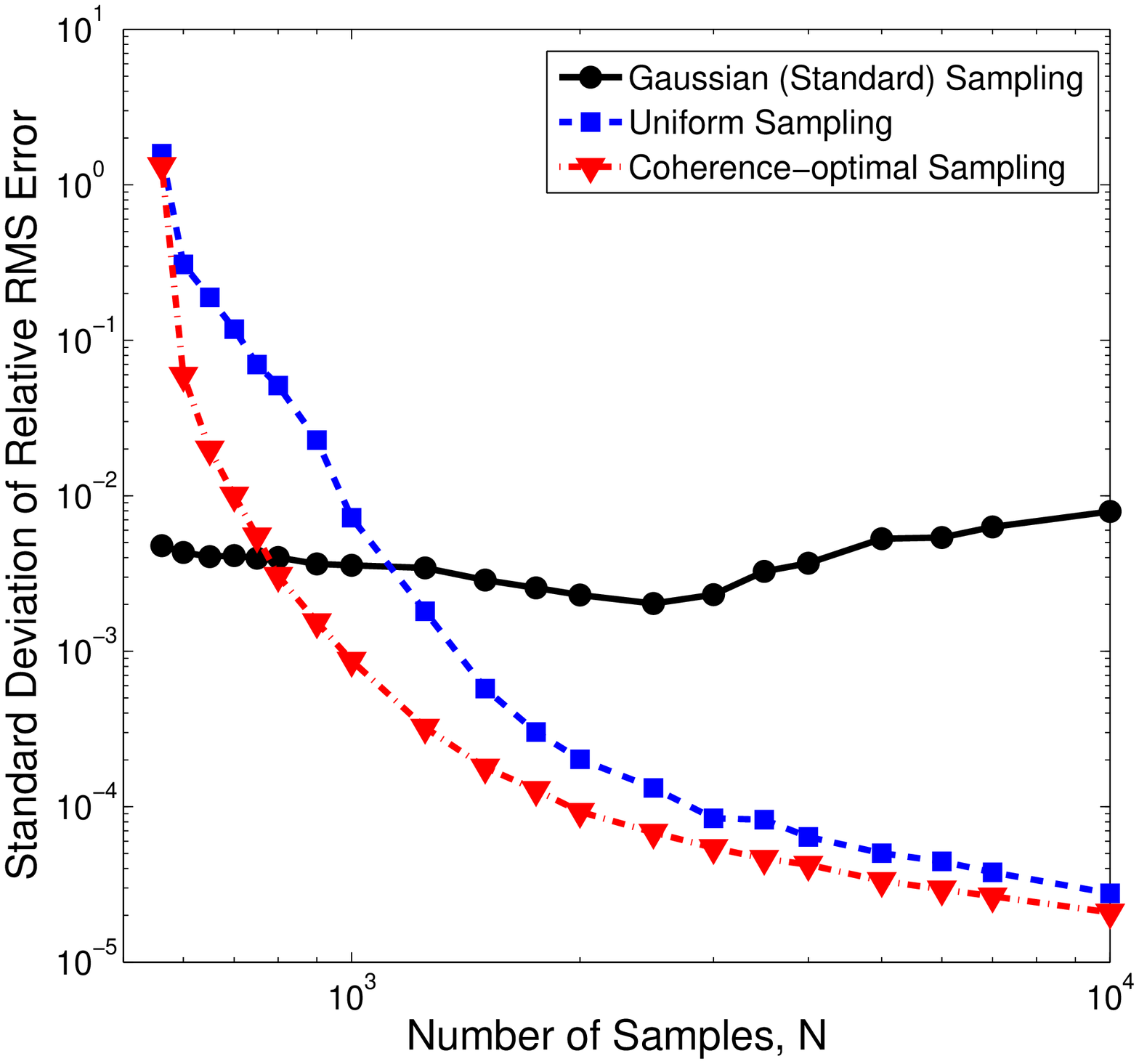}
\end{minipage}
\caption{Moments of root-mean-squared error between the reference and $\ell_2$-minimization solutions for the various sampling methods as a function of the number of samples.}
\label{Fig:ODERMSE}
\end{figure}

\subsection{Thermally Driven Cavity Flow}
\label{subsec:Heat}

Following \cite{LeMaitre02b,LeMaitre10,LeQuere91,Peng14}, we next consider a $2$-D heat driven square cavity flow problem with uncertain wall temperature, as shown in Fig. \ref{fig:cavity}. The left vertical wall has a deterministic, constant temperature $\tilde{T}_h$, referred to as the hot wall, while the right vertical wall has a stochastic temperature $\tilde T_c<\tilde T_h$ with constant mean $\bar{\tilde T}_c$, referred to as the cold wall. Both the top and bottom walls are assumed to be adiabatic. The reference temperature and the reference temperature difference are defined as $\tilde T_{ref}=(\tilde{T}_h+\bar{\tilde T}_c)/2$ and $\Delta \tilde T_{ref}=\tilde{T}_h-\bar{\tilde T}_c$, respectively. Under the assumption of small temperature differences, i.e., Boussinesq approximation, the governing equations in dimensionless variables are given by
\begin{equation}
\begin{aligned}
&\frac{\partial \bm{u}}{\partial t} + \bm{u}\cdot\nabla\bm{u}=-\nabla p + \frac{\text{Pr}}{\sqrt{\text{Ra}}}\nabla^2\bm{u}+\text{Pr}T\bm{\hat{y}},\\
& \nabla\cdot\bm{u}=0,\\ \label{eqn:cavity}
&\frac{\partial T}{\partial t}+ \nabla\cdot(\bm{u}T)=\frac{1}{\sqrt{\text{Ra}}}\nabla^2T,
\end{aligned}
\end{equation}
where $\bm{\hat{y}}$ is the unit vector $(0,1)$, $\bm{u}=(u,v)$ is velocity vector field, $T=(\tilde{T}-\tilde T_{ref})/\Delta \tilde T_{ref}$ is normalized temperature ($\tilde{T}$ denotes non-dimensional temperature), $p$ is pressure, and $t$ is time. Non-dimensional Prandtl and Rayleigh numbers are defined, respectively, as $\text{Pr}={\tilde \mu}\tilde c_p/\tilde \kappa$ and $\text{Ra}={\tilde \rho}{g}\beta\Delta \tilde T_{ref}{\tilde L}^3/({\tilde \mu}{\tilde \kappa})$, where the superscript tilde ($\tilde{~}$) denotes the non-dimensional quantities. Specifically, ${\tilde \mu}$ is molecular viscosity, ${\tilde \kappa}$ is thermal diffusivity, ${\tilde \rho}$ is density, ${g}$ is gravitational acceleration, the coefficient of thermal expansion is given by $\beta$, and ${\tilde L}$ is the reference length. In this example, the Prandtl and Rayleigh numbers are set to $\text{Pr}=0.71$ and $\text{Ra}=10^6$, respectively. For more details on the non-dimensional variables in (\ref{eqn:cavity}), we refer the 
reader to \cite{LeQuere91,LeMaitre02b,LeMaitre10}. 
On the cold wall, we use a (normalized) temperature distribution with stochastic fluctuations of the form
\begin{equation}
\begin{aligned}
T_c(x=1,y,\bm\Xi)=\bar{T}_c + T_c',\\
T_c'=\sigma_T\sum_{i=1}^{d}\sqrt{\lambda_i}\varphi_i(y)\Xi_i,
\end{aligned}
\label{eqn:coldwall}
\end{equation}
where $\bar{T}_c$ is a constant mean temperature. In (\ref{eqn:coldwall}), $\Xi_i$, $i=1,\dots,d$, are independent random variables uniformly distributed on $[-1,1]$. $\{\lambda_i\}_{i=1}^d$ and $\{\varphi_i(y)\}_{i=1}^d$ are the $d$ largest eigenvalues and the corresponding eigenfunctions of the exponential covariance kernel
%
%\begin{equation*}
$C_{T_cT_c}(y_1,y_2) = \exp{\left(-\frac{\vert y_1-y_2\vert}{l_c}\right)}$,
%\end{equation*}
%
where $l_c$ is the correlation length. In this setting, a (semi-)analytic representation of the eigenpairs $(\lambda_i,\varphi_i(y))$ in~(\ref{eqn:coldwall}) is available, see, e.g.,~\cite{Ghanem03}.

In our numerical tests, we let $(T_h,\bar T_c)=(0.5,-0.5)$, $d=20$, $l_c=1/21$, and $\sigma_T=11/100$. Our QoI, the vertical velocity component at $(x,y)=(0.25,0.25)$ denoted by $v(0.25,0.25)$, is expanded in the Legendre PC basis of total degree $p=4$ with only the first $P=2500$ basis functions retained. We seek to accurately reconstruct $v(0.25,0.25)$ with $N< P$ random samples of $\bm{\Xi}$ and the corresponding realizations of $v(0.25,0.25)$. These samples are constructed from a least squares reference solution constructed from these 2500 basis polynomials as computed from a large number of Monte Carlo samples.

\begin{figure}[ht]
\centering
\includegraphics[height=0.35\textheight]{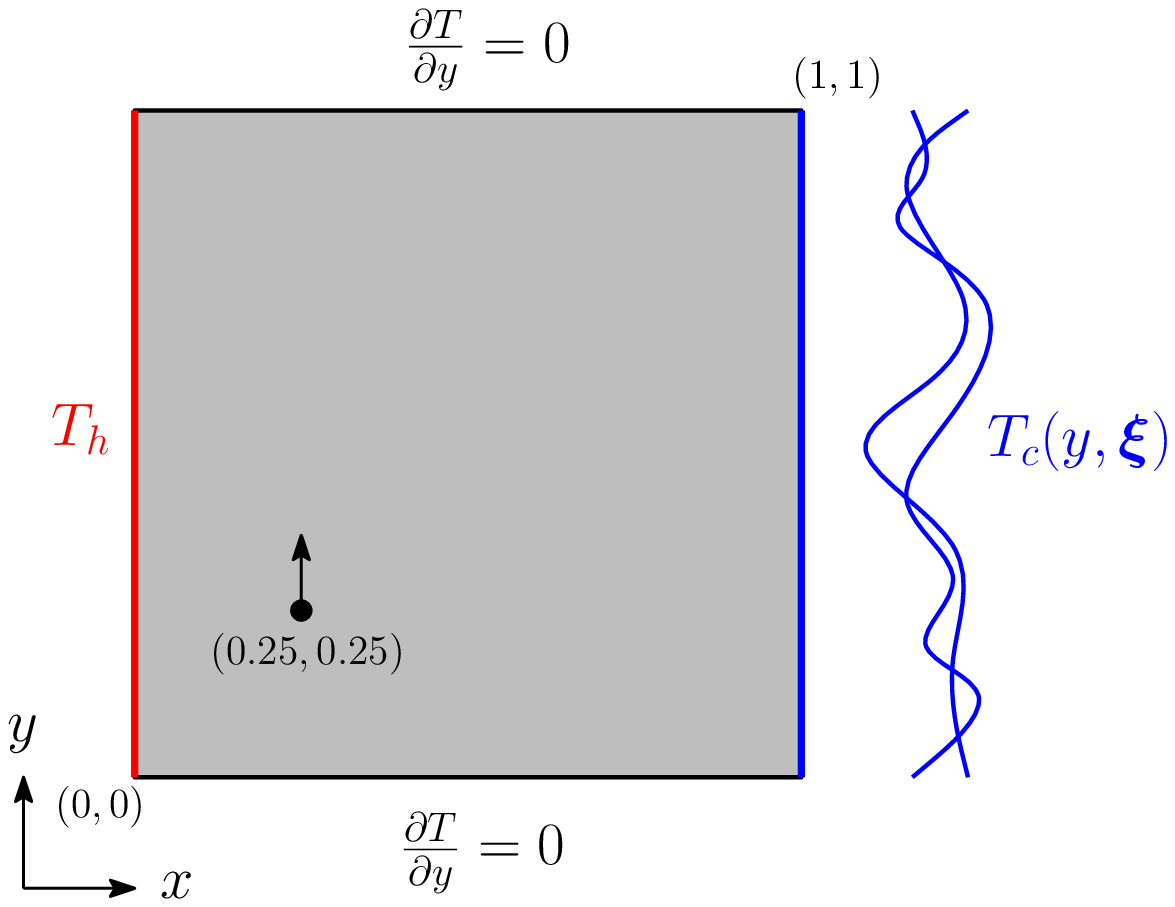}
\caption{Illustration of the cavity flow problem.}
\label{fig:cavity}
\end{figure}

To identify $u(\bm{\Xi})$ as a function of the random inputs $\bm{\Xi}$, we use a Legendre PC expansion of total order $p=3$, which for this $d=20$ stochastic dimensional problem yields $P=1771$ basis functions. We note that the root-mean-squared error is considered here as the primary measure of recovery, and that this leaves a truncation error from the 729 basis functions not used in the recovery.

We investigate the ability to recover $u(\bm{\Xi})$ via (\ref{eqn:lsqr}), using each of the three sampling schemes considered for Legendre polynomials. For this problem we further improve the quality of the MCMC sampling through an initial burn-in of 1,000 discarded samples~\cite{MCMCBook}. We provide bootstrapped estimates of the various moment based measures from a pool of samples generated beforehand. Specifically, samples for each realization are drawn from a pool of 50,000 previously generated samples, which are used to calculate bootstrap estimates of averages and standard deviations.

In Figure~\ref{Fig:EllipticSD}, we see plots of computed moments for the distribution of the relative root-mean-squared error between the computed and reference solutions obtained from 100 independent replications for each sample size, $N$. We note the standard and coherence-optimal sampling offer significant improvements over asymptotic, i.e., Chebyshev, sampling using similar sample sizes, $N$, both in terms of accuracy and robustness to differing realized samples. These observations are compatible with the theoretical results of Section \ref{sec:sampling} demonstrating a smaller coherence for the uniform sampling -- as compared to Chebyshev sampling -- for the case of $d>p$. The coherence-optimal sampling by construction leads to the theoretically smallest coherence, and provides a useful comparison for the other sampling schemes.

\begin{figure}[h!]
\hspace{-.55cm}
\begin{minipage}{.5\textwidth}
\centering
\includegraphics[scale=0.35]{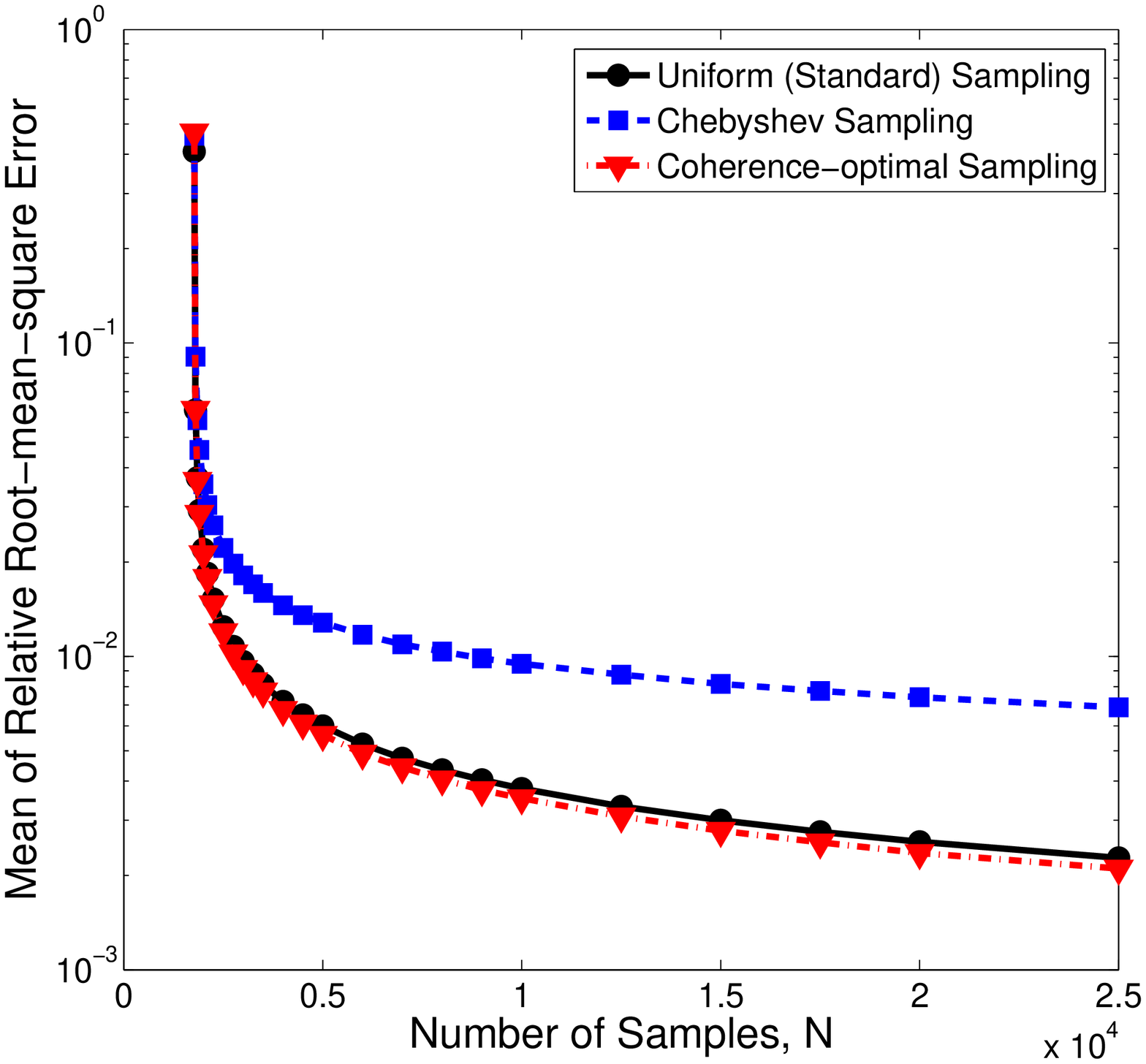}
\end{minipage}
\hspace{.1cm}
\begin{minipage}{.5\textwidth}
\centering
\includegraphics[scale=0.35]{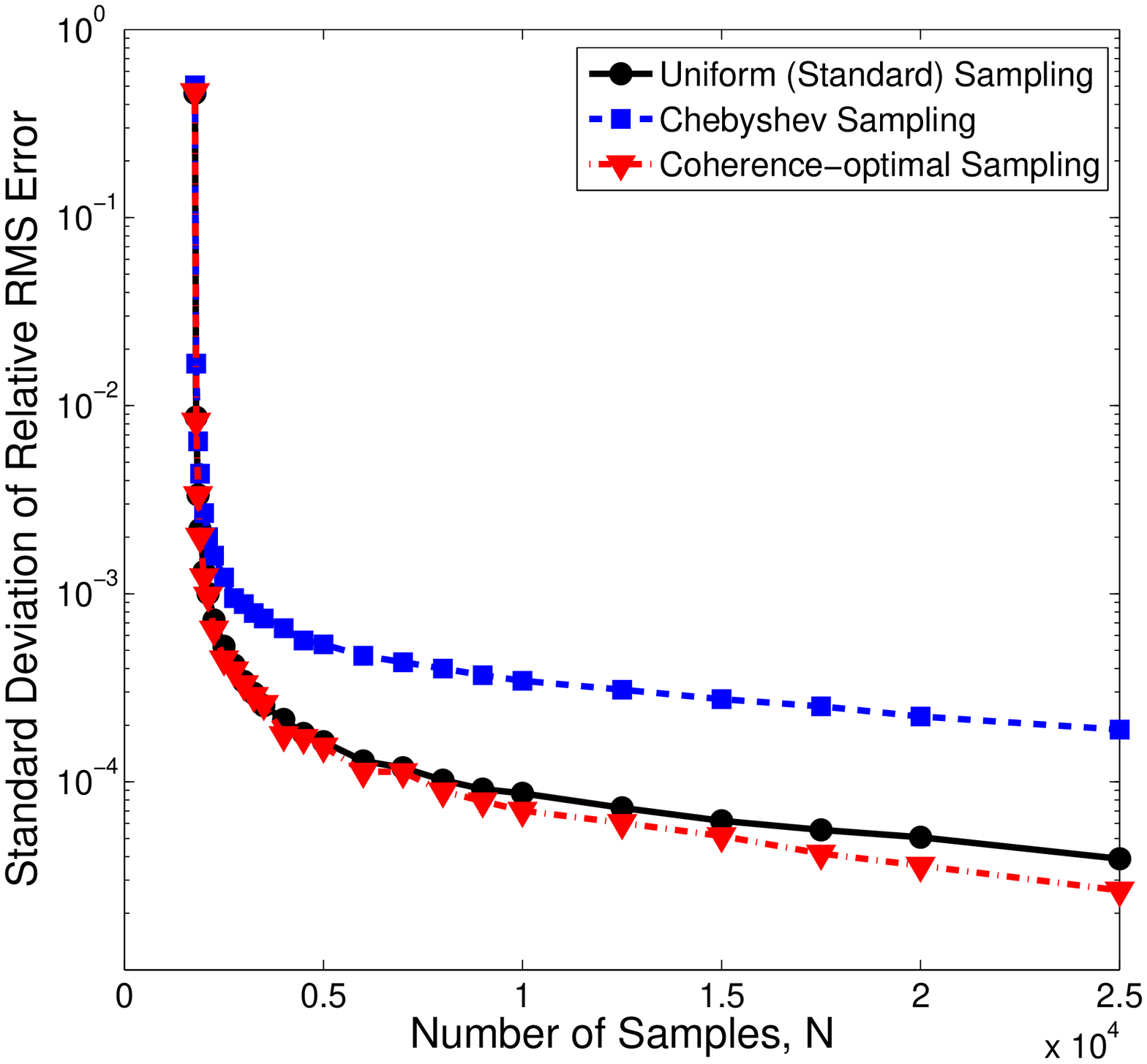}
\end{minipage}
\caption{Plots for the moments of root-mean-squared error for independent residuals for the various sampling methods as a function of the number of samples.}
\label{Fig:EllipticSD}
\end{figure}
\section{\texorpdfstring{Proofs of Theorems in Section~\ref{sec:motivation}}{Proofs of Theorems in Section~\ref{sec:motivation}}}
\label{sec:Proofs}

We now present proofs for the theorems in Section~\ref{sec:motivation}, where we use the truncation of $\Omega$ to $\mathcal{S}$, and believe it is clear how to alter proofs to analogous results when this truncation is not made. We note that the recovery from the least squares problem (\ref{eqn:lsqr}) is connected with the spectral convergence of
\begin{align*}
\bm{M}:=\frac{1}{N}\left(\bm{W\Psi}\right)^T\bm{W\Psi}
\end{align*}
to the identity matrix, see, e.g.,~\cite{CDL13} for $\bm{W}=\bm{I}$. To insure stability, we wish to guarantee that the largest singular value of $\bm{M}-\bm{I}$, denoted by $\sigma$, is in the interval $[0,1/2]$ with high probability. Indeed, we will consider any $\bm{M}$ that leads to $\sigma>1/2$ as unstable, and consider the entire sampling a failure. As a result, we first wish to demonstrate that this failure probability is low, noting the immediate applicability from Theorem 1 of~\cite{CDL13} when (\ref{Eqn:CoherenceBounded}) is used. Throughout we use the restricted expectation notation as defined in (\ref{eqn:restricted_expectation}), which is more convenient than the closely related conditional expectation. The following theorems also depend on a sampling event which we denote $\mathcal{E}$ that is the intersection of two events. 

The first is that all $N$ samples of $\bm{Y}$ would fall within $\mathcal{S}^c$, and consider any such realizations to lead to a failure for the entire sample. We may thus bound $\mu_{2}$ only over a set $\mathcal{S}$. Using the definition in (\ref{Eqn:CoherenceEll2Unbounded}), we can guarantee that each realization only fails with probability at most $1/(NP)$, and so via the union bound, this failure probability is no larger than $1/P$, which is the value that is seen in our proofs. While this term is important when considering e.g. standard sampling of Hermite polynomials where there is a probability that $\bm{\Xi}\in\mathcal{S}^c$, practically, the sampling distribution for $\bm{Y}$ can be supported only $\mathcal{S}$, such as via a truncated normal distribution, and this term may be dropped. We note that while we use $1/P$, tight bounds or exact computations of this failure probability are possible, which will typically make this term far smaller.

The second event which we will require for $\mathcal{E}$ is that the matrix $\bm{M}$ is close enough to identity so that the spectral norm $\|\bm{M}-\bm{I}\|<1/2$, an event which we will investigate more thoroughly.

We will use the condition,
\begin{align*}
\sum_{k=1}^{P}\mathbb{E}\left[|w(\bm{Y})\psi_k(\bm{Y})|^2\bm{1}_{\mathcal{S}^{c}}\right]\le\frac{1}{20\sqrt{P}},
\end{align*}
to insure that $\mathcal{S}^c$ is not important in the $\ell_2$ and, as we will now investigate, a spectral sense. The following lemma relates the convergence of $\bm{M}$ to $\bm{I}$ to a bound on the $\ell_2$-norm of the random rows. It is related to Lemma 1 of~\cite{CDL13}, and Section B of~\cite{CandesPlan}.
\begin{lem}
 \label{lem:SpectrumChernoffBound}
Let 
\begin{align*}
b&:=\sum_{k=1}^{P}\mathbb{E}\left[|w(\bm{Y})\psi_k(\bm{Y})|^2\bm{1}_{\mathcal{S}^{c}}\right]
\end{align*}
denote the bias in $\ell_2$-variation that is lost by restricting our sampled set to $\mathcal{S}$. Stated another way, if we let
\begin{align*} %Jerrad: I removed the incorrect normalizing constant 1/sqrt{N} which shows up later, not here.
(\bm{w\psi})_i&:= \left(w(\bm{Y}_i)\psi_1(\bm{Y}_i),\cdots w(\bm{Y}_i)\psi_P(\bm{Y}_i)\right)^T,
\end{align*}
then
\begin{align*}
\mathbb{E}\left((\bm{w\psi})_i^T(\bm{w\psi})_i;\mathcal{S}\right) &= P-b.
\end{align*}
If $\|\cdot\|$ denotes the spectral norm of the matrix, and $N\sqrt{P}\ge 20$, then
\begin{align*}
 \mathbb{P}\Bigg(\|\bm{M}-\bm{I}\|\le 1/2\Bigg| \bm{Y}_i\in\mathcal{S}\ \forall i\Bigg)&\le 2P\exp(-0.1N\mu_{2}^{-1}).
\end{align*}
\end{lem}
\begin{proof}
 We have that
\begin{align*}
\bm{I} &=\mathbb{E}\left((\bm{w\psi})(\bm{w\psi})^T|\mathcal{S}\right)\mathbb{P}(\mathcal{S})+\mathbb{E}\left((\bm{w\psi})(\bm{w\psi})^T|\mathcal{S}^c\right)\mathbb{P}(\mathcal{S}^c).
\end{align*}
Letting $\|\cdot\|$ denote the spectral norm of the matrix. A brief calculation gives that
\begin{align*}
\left\|\mathbb{E}\left((\bm{w\psi})(\bm{w\psi})^T|\mathcal{S}\right)-\bm{I}\right\|&\le\frac{\mathbb{P}(\mathcal{S}^c)}{\mathbb{P}(\mathcal{S})}\left(\left\|\mathbb{E}\left((\bm{w\psi})(\bm{w\psi})^T\bigg|\mathcal{S}^c\right)\right\| + 1\right).
\end{align*}
Using Jensen's inequality we have that,
\begin{align*}
\left\|\mathbb{E}\left((\bm{w\psi})(\bm{w\psi})^T|\mathcal{S}\right)-\bm{I}\right\|&\le\frac{\mathbb{P}(\mathcal{S}^c)}{\mathbb{P}(\mathcal{S})}\left[\mathbb{E}\left(\|\bm{w\psi}\|^2\bigg|\mathcal{S}^c\right) + 1\right].
\end{align*}
From the conditions in (\ref{Eqn:CoherenceEll2Unbounded}), this can be bounded by
\begin{align*}
\left\|\mathbb{E}\left((\bm{w\psi})(\bm{w\psi})^T|\mathcal{S}\right)-\bm{I}\right\|&\le\frac{1}{1-(NP)^{-1}}\left[\frac{1}{20\sqrt{P}} + \frac{1}{NP}\right].
\end{align*}
If $N\sqrt{P}\ge 20$, then this is bounded by
\begin{align*}
\left\|\mathbb{E}\left((\bm{w\psi})(\bm{w\psi})^T|\mathcal{S}\right)-\bm{I}\right\|&\le\frac{1}{10\sqrt{P}}=:\epsilon,
\end{align*}
which should not be confused with a truncation error that also utilizes $\epsilon$. 
We note that for this proof $\epsilon$ is not related to truncation error, but rather the bias in spectral convergence as a result of truncation of $\Omega$ to $\mathcal{S}$ in (\ref{Eqn:CoherenceEll2Unbounded}). This insures that if our samples come from $\mathcal{S}$, then the mean sample is nearly orthonormal, that is there is not much bias from this truncation. We utilize a Chernoff bound to show that there is a fast convergence to this mean.
We have that
\begin{align*} 
1-\epsilon\le\lambda_{\min}\left(\mathbb{E}(\bm{M})\right)\le\lambda_{\max}\left(\mathbb{E}(\bm{M})\right)\le1+\epsilon.
\end{align*}
An application of the Chernoff bound as in Theorem 1.1 of \cite{Tropp12a} and Theorem 1 of \cite{CDL13} gives that for $\delta \in[0,1]$
\begin{align*}
 \mathbb{P}\Bigg(\lambda_{\min}\left(\bm{M}\right)\le (1-\delta)(1-\epsilon)\Bigg| \bm{Y}_i\in\mathcal{S}\ \forall i\Bigg)&\le P\left(\frac{e^{-\delta}}{(1-\delta)^{1-\delta}}\right)^{N(1-\epsilon)\mu_2^{-1}};\\ 
\mathbb{P}\Bigg(\lambda_{\max}\left(\bm{M}\right)\ge (1+\delta)(1+\epsilon)\Bigg| \bm{Y}_i\in\mathcal{S}\ \forall i\Bigg)&\le P\left(\frac{e^{\delta}}{(1+\delta)^{1+\delta}}\right)^{N(1+\epsilon)\mu_2^{-1}}.  %%Alireza: Should (1+\epsilon) be the exponent of the RHS? %% JH: Fixed.
\end{align*}
Note that
\begin{align*}
 (1-\delta)(1-\epsilon)\le1/2 &\implies \delta \ge \frac{1-2\epsilon}{2-2\epsilon};\\
 (1+\delta)(1+\epsilon)\ge 3/2 &\implies \delta \ge \frac{1-2\epsilon}{2+2\epsilon}. 
\end{align*}
and so we have that a critical $\delta$, denoted $\delta_{\star}$, such that for all $\delta<\delta_{\star}$ the matrix $\bm{M}$ is guaranteed to satisfy $\|\bm{M}-\bm{I}\|_2\le1/2$ is $\delta_\star:=(1-2\epsilon)/(2+2\epsilon)$ as it is the smaller tolerance. Note that for $0\le\delta<1$,~\cite{CDL13},
\begin{align*}
 \frac{e^{-\delta}}{(1-\delta)^{1-\delta}}\ge\frac{e^{\delta}}{(1+\delta)^{1+\delta}},
\end{align*}
and so we may bound the sum of the probabilities by
\begin{align*}
  \mathbb{P}\Bigg(\|\bm{M}-\bm{I}\|\le \frac{1}{2}\Bigg| \bm{Y}_i\in\mathcal{S}\ \forall i\Bigg)\le 2P\left(\frac{e^{-\delta_\star}}{(1-\delta_\star)^{1-\delta_\star}}\right)^{N(1-\epsilon)\mu_2^{-1}}.
\end{align*}
We now bound this probability. We note that if $\epsilon\le1/2$ then
\begin{align*}
\delta_\star\le\frac{1-2\epsilon}{2+2\epsilon}\le \frac{1}{2}-\epsilon.
\end{align*}
To create a useful bound in terms of $\epsilon$, we note that
\begin{align*}
c_{\epsilon}&:=1/2-\epsilon+ (1/2+\epsilon)\log(1/2+\epsilon);\\
 \frac{e^{-\delta_\star}}{(1-\delta_\star)^{1-\delta_\star}}&\le\exp(-c_{\epsilon}).
\end{align*}
Thus, noting that if $P=1$, then $\epsilon = 1/20$ and $c_\epsilon\approx 0.1212$,
\begin{align*}
\mathbb{P}\Bigg(\|\bm{M}-\bm{I}\|\le \frac{1}{2}\Bigg| \bm{Y}_i\in\mathcal{S}\ \forall i\Bigg)&\le 2P\exp\left(-c_{\epsilon}(1-\epsilon)N\mu_{2}^{-1}\right);\\
&\le 2P\exp\left(-0.95c_{\epsilon}N\mu_{2}^{-1}\right);\\
&\le 2P\exp(\left(-0.114N\mu_{2}^{-1}\right).
\end{align*}
As $P\rightarrow\infty$, $c_\epsilon\rightarrow 0.1534$, and so to does the multiplicative constant which is seen to not heavily depend on $P$.
\end{proof}
This Lemma shows that with a probability that converges to 1 exponentially in $N/\mu_{2}$, $\bm M$ is well-conditioned. Changes could allow us to lower the bound on $\|\bm{M}-\bm{I}\|$ from $1/2$ at the cost of a higher failure probability. This reduction on $\|\bm{M}-\bm{I}\|$ will be useful in a way that we see in the next proof. Also we could alter the conditions of (\ref{Eqn:CoherenceEll2Unbounded}) for $\mathcal{S}$ to affect this probability. We consider neither point here, but we refer the interested reader to Theorem 1 of~\cite{CDL13} and Section B of~\cite{CandesPlan}.

We are now equipped to prove Theorem~\ref{thm:SampleDepth} and~\ref{thm:SampleDepthNoise}, but before doing so we first briefly show that the truncation error that is orthogonal to the basis functions restricted to the smaller set $\mathcal{S}$ is less than that over the full set $\Omega$. 

\begin{lem}
\label{lem:TruncLemma}
Let $\epsilon_\mathcal{S}(\bm{\xi})$ be the truncated error that is orthogonal to the basis functions restricted to $\mathcal{S}$, and let $\epsilon_\Omega(\bm{\xi})$ be the error that is orthogonal to the basis functions over the full $\Omega$. Then
\begin{align*}
\mathbb{E}(\epsilon_{\mathcal{S}}^2(\bm{\Xi});\mathcal{S})\le\mathbb{E}(\epsilon_{\Omega}^2(\bm{\Xi})).
\end{align*}
\end{lem}
We use this inequality in the proofs of Theorem~\ref{thm:SampleDepth} and~\ref{thm:SampleDepthNoise}, as our errors are based on the truncated orthogonal version, and we would like to substitute for the larger, more global error. We note that this theorem applies only to the theoretically optimal solutions, and not any computed approximation.

\begin{proof}
Let $\mathcal{V}$ be the span of $\{\psi_1(\bm{\xi}),\cdots,\psi_P(\bm{\xi})\}$. We consider the least squares
 solution over $\mathcal{S}$, and denote the $L_2(\mathcal{S},f)$-optimal approximation to $u$ as 
\begin{align*}
 \hat{u}_\mathcal{S} :=\mathop{\arg\min}\limits_{u_\star\in\mathcal{V}}\int_{\mathcal{S}}|u(\bm{\xi})-u_\star(\bm{\xi})|^2f(\bm{\xi})d\bm{\xi}.
\end{align*}
Similarly, we let $\hat{u}_\Omega$ denote the least squares
 solution that minimizes the above over $\Omega$. It follows that
\begin{align*} 
\int_{\mathcal{S}}|u(\bm{\xi})-\hat{u}_\mathcal{S}(\bm{\xi})|^2f(\bm{\xi})d\bm{\xi}&\le\int_{\mathcal{S}}|u(\bm{\xi})-\hat{u}_\Omega(\bm{\xi})|^2f(\bm{\xi})d\bm{\xi},\\
&\le \int_{\Omega}|u(\bm{\xi})-\hat{u}_\Omega(\bm{\xi})|^2f(\bm{\xi})d\bm{\xi}.
\end{align*}
\end{proof}
\subsection{\texorpdfstring{Proof of Theorem~\ref{thm:SampleDepth}}{Proof of Theorem~\ref{thm:SampleDepth}}}
\label{subsec:proofEll2TruncOnly}

We recall that our error is conditioned on a set $\mathcal{E}$ which is the intersection of two events. The first event, $\mathcal{E}_1$, is that all sampled $\bm{Y}$ are in the set $\mathcal{S}$, which may be a proper subset of $\Omega$. The second event, $\mathcal{E}_2$, is that $\|\bm{M}-\bm{I}\|\le 1/2$. We note that the probability density function $f(\bm{\xi})$ for $\bm{\Xi}$ is related to that for $\bm{Y}$ multiplied by $w^2(\bm{\xi})$. We are premultiplying by $w(\bm{Y})$ because of sampling $\bm{Y}$ instead of $\bm{\Xi}$, and so we are able to write expectations of $w^2(\bm{Y})g(\bm{Y})$ in terms of $\bm{\Xi}$. Further, for any appropriately measurable function $g$ and appropriately measurable set $\mathcal{T}$,
\begin{align*}
\mathbb{E}(w^2(\bm{Y})g(\bm{Y});\mathcal{T})=\mathbb{E}(g(\bm{\Xi});\bm{\Xi}\in\mathcal{T}).
\end{align*}
As the event $\mathcal{E}$ includes that all samples are in $\mathcal{S}$, we thus have that
\begin{align*}
\mathbb{E}(w^2(\bm{Y})g(\bm{Y});\mathcal{E})=\mathbb{E}(g(\bm{\Xi});\mathcal{E}).
\end{align*}
The remainder of the proof borrows methodology from that of Theorem 2 of~\cite{CDL13}, while having several differences regarding the truncation of $\Omega$ to $\mathcal{S}$. Recalling that $\epsilon(\bm{\Xi})$ is the truncation error, we have that $w^2(\bm{Y})\epsilon(\bm{Y})$ is orthogonal to the space $\mathcal{V}=\mbox{span}\left\{\psi_1(\bm{\Xi}),\cdots,\psi_P(\bm{\Xi})\right\}$. We may define $\hat{\bm{a}}$ to be the difference between the true theoretical solution and our computed solution by
\begin{align*}
\hat{\bm{a}}:=\bm{c}-\hat{\bm{c}}
\end{align*}
and note that this implies that
\begin{align*}
\bm{M}\hat{\bm a}&=\bm{b},
\end{align*}
where, as the rows of $\bm{u}$ and $\bm{\Psi}$ are each multiplied by $w(\bm{Y}_i)$ and $\hat{\bm{c}}$ is a least-squares solution,
\begin{align*}
b_k&=\frac{1}{N}\mathop{\sum}\limits_{i=1}^Nw^2(\bm{Y}_i)\epsilon(\bm{Y}_i)\psi_k(\bm{Y}_i).
\end{align*}
As we assume the hypotheses of Lemma~\ref{lem:SpectrumChernoffBound}, we have that $\|\bm{M}^{-1}\|\le 2$, and thus
\begin{align*}
\|\hat{\bm a}\|_2^2\le 4\|\bm{b}\|_2^2.
\end{align*}
Applying the expectation, integral, and a triangular inequality gives that %%Alireza: shall we present this also is a new format used in Jerrad: I feel this thought wasn't completed. I did make some changes in the following presentation for clarity.
\begin{align*}
\mathbb{E}\left(\int_{\Omega}(u(\bm{\xi})-\hat{u}(\bm{\xi}))^2f(\bm{\xi})d\bm{\xi};\mathcal{E}\right)&\le \mathbb{E}(\epsilon(\bm{\Xi})^2)+4\mathbb{E}(\|\bm{b}\|_2^2;\mathcal{E}).
\end{align*}
We now work to bound $\mathbb{E}(\|\bm{b}\|_2^2;\mathcal{E})$, noting that this includes that all draws are restricted to come from $\mathcal{S}$. We have that
\begin{align*}
\mathbb{E}(b_k^2;\mathcal{E})&\le \frac{1}{N^2}\mathop{\sum}\limits_{i=1}^N\mathop{\sum}\limits_{j=1}^N\mathbb{E}(w^2(\bm{Y}_i)\epsilon(\bm{Y}_i)\psi_k(\bm{Y}_i)w^2(\bm{Y}_j)\epsilon(\bm{Y}_j)\psi_k(\bm{Y}_j));\mathcal{S}),\\
&\le\frac{1}{N^2}\mathop{\sum}\limits_{i=1}^N\mathbb{E}(w^4(\bm{Y}_i)\epsilon^2(\bm{Y}_i)\psi^2_k(\bm{Y}_{i});\mathcal{S}),\\
&\le\frac{1}{N}\mathbb{E}(w^2(\bm{Y})\epsilon^2(\bm{Y});\mathcal{S})\mathop{\sup}\limits_{\bm{\xi}\in\mathcal{S}}|w^2(\bm{\xi})\psi^2_k(\bm{\xi})|,
\end{align*}
where the first inequality comes from the definition of $b_k$; the second from the independence of the $\bm{Y}_k$ and orthogonality of $\{\psi_k\}_{k=1}^P\cup\epsilon$; and the third is a direct point-wise bound. Hence, summing for $k=1:P$ and recalling the definition in (\ref{Eqn:CoherenceEll2Unbounded}), we arrive at
\begin{align*}
\mathbb{E}(\|\bm{b}\|_2^2;\mathcal{E})&\le\frac{\mathbb{E}(\epsilon^2(\bm{\Xi}))}{N}\cdot\mathop{\sup}\limits_{\bm{\xi}\in\mathcal{S}}\mathop{\sum}\limits_{k=1}^P|w(\bm{\xi})\psi_k(\bm{\xi})|^2,\\
&\le\frac{\mathbb{E}(\epsilon^2(\bm{\Xi}))}{N}\cdot\mu_{2}.
\end{align*}
Putting the above pieces together, we have that
\begin{align*}
\mathbb{E}\left(\int_{\Omega}(u(\bm{\xi})-\hat{u}(\bm{\xi}))^2f(\bm{\xi})d\bm{\xi};\mathcal{E}\right)&\le \mathbb{E}(\epsilon^2(\bm{\Xi}))\left(1+\frac{4\mu_{2}}{N}\right).
\end{align*}
This completes the proof.

\subsection{\texorpdfstring{Proof of Theorem~\ref{thm:SampleDepthNoise}}{Proof of Theorem~\ref{thm:SampleDepthNoise}}}
\label{subsec:proofEll2FullError}
This proof is related to Theorem 3 of~\cite{CDL13}, although we bound the measurement error similarly to the bound on truncation error. We let $\epsilon_T$ denote the truncation error (denoted by $\epsilon$ in the previous theorem), with corresponding hat notation. Additionally we let $\epsilon_M(\bm{Y}):=\mathbb{E}(\epsilon_M(\bm{\chi};\bm{Y})|\bm{Y})$, be itself a random quantity that representing the contribution arising from the unobserved random vector, $\bm{\chi}$.  Similarly to before, we define $\hat{\bm{a}}_m$ to be the true solution, minus the computed solution as well as the contribution defined above that arises from error due to the truncation of the basis, leaving only a solution computation associated with measurement error. That is,
\begin{align*}
\hat{\bm{a}}_m:=\bm{c}-\hat{\bm{c}}-\hat{\bm{a}}
\end{align*}
is the portion of the solution which fits the measurement noise and satisfies
\begin{align*}
\bm{M}\hat{\bm{a}}_m=\bm{d}.
\end{align*}
Here, $\bm{d}$ is given by
\begin{align*}
d_k&=\frac{1}{N}\mathop{\sum}\limits_{i=1}^Nw^2(\bm{Y}_i)\epsilon_M(\bm{Y}_i)\psi_k(\bm{Y}_i).
\end{align*}
Applying the integral and expectation similarly to before we have that,
\begin{align}
\label{eqn:pluginforsum}
\mathbb{E}\left(\int_{\Omega}(u(\bm{\xi})-\hat{u}(\bm{\xi}))^2f(\bm{\xi})d\bm{\xi};\mathcal{E}\right)&\le \mathbb{E}(\epsilon^2(\bm{\Xi}))+4\cdot\left[2\mathbb{E}\left(\|\bm{b}\|_2^2;\mathcal{E}\right)+2\mathbb{E}\left(\|\bm{d}\|_2^2;\mathcal{E}\right)\right].
\end{align}
We bound $\mathbb{E}(d^2_k;\mathcal{E})$ similarly to how we bounded $\mathbb{E}(b^2_k;\mathcal{E})$. Particularly,
\begin{align*}
\mathbb{E}(d_k^2;\mathcal{E})&\le \frac{1}{N^2}\mathop{\sum}\limits_{i=1}^N\mathop{\sum}\limits_{j=1}^N\mathbb{E}(w^2(\bm{Y}_i)\epsilon_M(\bm{Y}_i)\psi_k(\bm{Y}_i)w^2(\bm{Y}_j)\epsilon_M(\bm{Y}_j));\mathcal{S}),\\
&\le\frac{1}{N^2}\mathop{\sum}\limits_{i=1}^N\mathbb{E}(w^4(\bm{Y}_i)\epsilon_M^2(\bm{Y}_i)\psi^2_k(\bm{Y}_{i});\mathcal{S}),\\
&\le\frac{1}{N}\mathbb{E}(w^2(\bm{Y})\epsilon_M^2(\bm{Y});\mathcal{S})\mathop{\sup}\limits_{\bm{\xi}\in\mathcal{S}}|w^2(\bm{\xi})\psi^2_k(\bm{\xi})|.
\end{align*}
Which gives that
\begin{align*}
\mathbb{E}(\|\bm{d}\|_2^2;\mathcal{E})&\le \frac{\mathbb{E}(\epsilon^2_M(\bm{\Xi}))\mu_{2}}{N}.
\end{align*}
Using this and the results from Theorem~\ref{thm:SampleDepth} plugged into (\ref{eqn:pluginforsum}), completes the proof.
 \subsection{\texorpdfstring{Proof of Theorem~\ref{thm:LowPTransformedSamples}}{Proof of Theorem~\ref{thm:LowPTransformedSamples}}}
 \label{subsec:proofEll2FullError}

While the first part of this follows from the proof of Theorem 4.5 of \cite{Hampton14}, we briefly prove the second statement, which arises from $B^2(\bm{\xi})$ being a linear combination of squared basis functions. We note that if we take $w(\bm{\xi})$ to be $1/B(\bm{\xi})$ then
 \begin{align*}
  \mu_2(\bm{Y})&:= \max_{\bm{\xi}\in\Omega}c^{-2}w^2(\bm{\xi})\mathop{\sum}\limits_{k=1}^P\psi^2_k(\bm{\xi}),\\
  &= \max_{\bm{\xi}\in\Omega}c^{-2} = c^{-2},
 \end{align*}
where it remains to identify the normalization constant $c^{-2}$ so that $f_{\bm{Y}}(\bm{\xi})= c^2f(\bm{\xi})/w^2(\bm{\xi})$ is a probability distribution. Recalling the definition of $w$ in this case we have that
 \begin{align*}
  1&=\int_{\bm{\xi}\in\Omega}f_{\bm{Y}}(\bm{\xi})d\bm{\xi},\\
  &= c^2\mathop{\sum}\limits_{k=1}^P\int_{\bm{\xi}\in\Omega}f(\bm{\xi})\psi^2_k(\bm{\xi})d\bm{\xi}.
  \end{align*}
 By the orthonormality of the $\{\psi_k(\bm\xi)\}_{k=1}^{P}$ it follows that $c^2=1/P$, implying that $\mu_2=c^{-2}=P$.

\section{\texorpdfstring{Conclusions}{Conclusions}}
\label{sec:conc}

We provided an analysis of Hermite and Legendre polynomials which allowed us to bound a coherence parameter and generate recovery guarantees for polynomial chaos expansions obtained via least squares regression. We also identified alternative random sampling schemes which provide sharper guarantees and demonstrate improved polynomial chaos reconstructions relative to the random sampling from the orthogonality measure of these bases. These sampling methods were derived based on the properties of Hermite and Legendre polynomials. Furthermore, we showed a Markov Chain Monte Carlo method for generating samples that minimizes the coherence parameter, thereby achieving an optimality for the number of random solution realizations. Such a sampling was referred to as coherence-optimal sampling, and guarantees recovery with a number of samples that scales linearly in the number of basis functions, up to log factors. Positive results were observed when computing the solution to a non-linear ordinary differential equation, where a high order Hermite polynomial chaos expansion was needed for an accurate solution approximation. Similarly, sampling methods were compared on arbitrary manufactured stochastic functions, and the different sampling strategies were tested for identifying the solution of a 20-dimensional heat driven fluid flow. In all examples, positive results were attained for the coherence-optimal sampling method.

\section*{Acknowledgements}
\label{sec:acknow}

We thank Ji Peng from CU Boulder for providing the deterministic codes for the problem in Section \ref{subsec:Heat}.

This material is based upon work supported by the U.S. Department of Energy Office of Science, Office of Advanced Scientific Computing Research, under Award Number DE-SC0006402, as well as National Science Foundation under grant DMS-1228359.

\bibliographystyle{elsarticle-num}
\bibliography{CoherencePolyCitations}
\end{document}